\documentclass[draft]{amsart}
\usepackage{amssymb}
\usepackage{latexsym}
\usepackage{amsfonts}
\usepackage{amsmath}
\usepackage{color}
\usepackage{longtable}
\usepackage{lscape}
\usepackage{makecell}
\usepackage{longtable}

\newtheorem{theorem}{Theorem}[section]
\newtheorem{proposition}[theorem]{Proposition}
\newtheorem{lemma}[theorem]{Lemma}
\newtheorem{corollary}[theorem]{Corollary}

\theoremstyle{definition}

\renewcommand{\leq}{\leqslant}
\renewcommand{\geq}{\geqslant}

\newcommand{\F}{\mathbb F}
\newcommand{\E}{\mathbb E}
\newcommand{\K}{\mathbb K}

\newcommand{\bft}{\mathbf t}
\newcommand{\bfx}{\mathbf x}

\newcommand{\ad}{{\rm ad}}
\newcommand{\g}{\mathfrak g}

\newcommand{\fracfield}[1]{\mbox{\rm Frac}(#1)}
\newcommand{\height}[1]{\mbox{ht}(#1)}
\newcommand{\spec}[1]{\mbox{Spec}(#1)}
\newcommand{\der}[1]{\mbox{\rm Der}(#1)}

\begin{document}

\title[The center of universal enveloping algebras]{The center of the universal enveloping algebras of small-dimensional nilpotent Lie algebras in prime characteristic}
\author{Vanderlei Lopes de Jesus}
\address[Lopes de Jesus]{Departamento de Matem\'atica\\
Instituto de Ci\^encias Exatas\\
Universidade Federal de Minas Gerais\\
Av.\ Ant\^onio Carlos 6627\\
Belo Horizonte, MG, Brazil. Email:
{\rm\texttt{vanderleilopesbh@gmail.com}}}

\author{Csaba Schneider}
\address[Schneider]{Departamento de Matem\'atica\\
Instituto de Ci\^encias Exatas\\
Universidade Federal de Minas Gerais\\
Av.\ Ant\^onio Carlos 6627\\
Belo Horizonte, MG, Brazil. Email:
{\rm\texttt{csaba@mat.ufmg.br}}, URL:
{\tt
  {https://schcs.github.io/WP/}}}

\begin{abstract}
    We describe the centers of the universal enveloping algebras of nilpotent Lie algebras of dimension at most six over fields of prime characteristic.  If the characteristic is not smaller than the nilpontency class, then the center is the integral closure of the algebra generated over the $p$-center by the same generators that also occur in characteristic zero. 
    Except for three examples (two of which are standard filiform), this algebra is already integrally closed and hence it coincides with the center. 
    In the case of these three exceptional algebras, the center has further generators. Then we show that the center of the universal enveloping algebra 
    of the algebras investigated in this paper is isomorphic to the Poisson center (the algebra of invariants under the adjoint representation). This shows that 
    Braun's conjecture is valid for this class of Lie algebras.
\end{abstract}

\date{5 November 2021}


\maketitle

\section{Introduction}\label{section1}

The aim of this paper is to calculate, as far as possible, explicit generators for the 
centers of the universal enveloping algebras of small-dimensional nilpotent Lie algebras
over fields of prime characteristic. Our objective is to determine such generators for
nilpotent Lie algebras of dimension less than or equal to 6 in characteristic $p$ where $p$ is not smaller than the nilpotency class.

This work is motivated partly by Braun's conjecture which states that if 
$\g$ is a finite-dimensional nilpotent Lie algebra, then the center $Z(\g)$ of 
the universal enveloping algebra is isomorphic, as a commutative algebra, to the 
algebra of polynomial invariants associated to the adjoint representation of $\g$ (see~\cite{ben-shimol}). 
This isomorphism, often referred to as the Duflo isomorphism, is well-known in characteristic zero over algebraically closed fields; 
see~\cite{Duflo,Kontsevich,dixmier96}. For finite dimensional nilpotent Lie algebras, 
the isomorphism between 
the two commutative algebras is given by the symmetrization map; see~\cite[Proposition~4.8.12]{dixmier96}. Ben-Shimol~\cite{ben-shimol}
verified Braun's conjecture for the nilpotent radicals of the Borel subalgebras of 
the classical simple Lie algebras of type $A$, $B$, $C$ and $D$.

In this paper we consider nilpotent Lie algebras of dimension at most $6$ which were classified by several authors
(see for example~\cite{deg,cgs12} and the references therein). Our methodology relies on the known structure theorems describing  the center $Z(\g)$ and the division algebra $D(\g)$ of quotients of the universal enveloping algebra $U(\g)$
(see Section~\ref{sec:center}) as well as on commutative algebra, more specifically, the theory of Cohen--Macaulay rings and 
integrally closed domains (see Section~\ref{sec:ca}). We were heavily influenced by earlier work of Ooms, 
Braun and Ben-Shimol~\cite{Ooms2,braun,braun2,braun-vernik,ben-shimol}.

Let $\g$ be a nilpotent Lie algebra of dimension at most six defined over a field of characteristic $p$ such that $p$ is not 
smaller than the nilpotency class of $\g$. 
Our first main result (Theorem~\ref{th1} in Section~\ref{sec:th1}) states that $Z(\g)$ is the integral closure of $Z_p(\g)[z_1,\ldots,z_k]$ where $Z_p(\g)$ 
is the the $p$-center of $U(\g)$ (see Section~\ref{sec:center}) and the $z_i$
are explicitly determined generators. The generators $z_i$ in this theorem are the same  
as the ones in characteristic zero (see~\cite{Ooms1,Ooms2}). 
Except for three algebras (namely, $\g_{5,5}$, $\g_{6,18}$ e $\g_{6,25}$ (two of which are standard filiform), see Section~\ref{sec:th1}), the center $Z(\g)$ 
actually coincides with $Z_p(\g)[z_1,\ldots,z_k]$. However, in the case of these three exceptional algebras, the algebra 
$Z_p(\g)[z_1,\ldots,z_k]$ is not integrally closed in its field of fractions, and hence it cannot be the whole center by Zassehaus' 
well-known theorem~\cite[Lemma~6]{zassenhaus53}. If $\g$ is one these three exceptional algebras, computations with Macaulay2 show
that the generating set of $Z(\g)$ depends on the characteristic $p$ and gets more and more complicated as $p$ grows. 

Having determined explicit generators for $Z(\g)$ (modulo integral closure), we consider polynomial invariants of $\g$ corresponding to 
the adjoint representations. The algebra $P(\g)$ of such invariants is often referred to as the Poisson center of $\g$. 
We review some properties of $P(\g)$ in Section~\ref{sec:pol_inv} and then in Section~\ref{sec:th2} we prove that for the algebras 
considered in the paper, the center $Z(\g)$ is isomorphic to the Poisson center $P(\g)$ (Theorem~\ref{th2}). Hence we verify Braun's conjecture for 
this class of algebras. 

For the reader's convenience, in the final Section~\ref{sec:tables} we summarize some properties of the nilpotent Lie algebras of dimension at most six in the form of tables.

\subsection*{Acknowledgment} The first author was financially 
supported by a PhD scholarship awarded by CNPq (Brazil). The second author acknowledges the financial support of  the CNPq projects 
\textit{Produtividade em Pesquisa} (project no.: 308212/2019-3)  
and \textit{Universal} (project no.: 421624/2018-3).  We thank Lucas Calixto, Letterio Gatto, Tiago Macedo, and Renato Vidal Martins  
for their  valuable comments.

\section{Some commutative algebra}\label{sec:ca}

In this section we review some facts, most of which are well-known, of commutative algebra that will be used in the rest of the paper.

\subsection{Inseparable extensions}
If $\mathbb{F}\subseteq \mathbb{E}$ is an extension of fields, then an $\F$-algebraic element $\alpha\in \mathbb{E}$
is {\em separable} over $\mathbb{F}$ if its minimal polynomial $\min_{\alpha}(x)\in\F[x]$ 
has no multiple roots in its splitting field over $\F$;  otherwise $\alpha$ is {\em inseparable}. An algebraic extension $\mathbb{F}\subseteq \mathbb{E}$ is {\em separable}, if all $\alpha\in \mathbb{E}$ is separable over $\mathbb{F}$. If 
no element of  $\mathbb{E} \setminus \mathbb{F}$ is separable over $\mathbb{F}$, then  $\mathbb{E}$ is {\em purely inseparable} over  $\mathbb{F}$.  
Note that a trivial extension $\mathbb{F}\subseteq\mathbb{F}$ is viewed as purely inseparable. 

\begin{theorem}[\cite{Isaacs94}, Theorem~19.10]\label{Isaa1}
 The following statements are equivalent for an algebraic extension 
 $\mathbb{F}\subseteq \mathbb{E}$ of fields 
 of  characteristic $p>0$.
 \begin{enumerate}
 \item $\mathbb{E}$ is purely inseparable over $\mathbb{F}$.
 \item For all $\alpha\in \mathbb{E}$, there exists $n\geq 0$ such that 
 $\alpha^{p^{n}}\in \mathbb{F}$.
 \item The minimal polynomial of all elements of $\mathbb{E}$ has the form $x^{p^{n}}-a$ for some integer $n\geq0$ and $a\in \mathbb{F}$.
 \end{enumerate}
\end{theorem}

\begin{corollary}[\cite{Isaacs94}, Corollaries~19.11-19.13]\label{cor:insep}
The following are valid for a field $\F$ of characteristic $p$.
\begin{enumerate}
\item Suppose that $\mathbb{E}=\mathbb{F}(\alpha)$ such that $\alpha^{p^{n}}\in \mathbb{F}$ for some $n\geq1$. Then $\mathbb{E}$ is purely inseparable over $\mathbb{F}$.
\item Suppose that $\mathbb{F}\subseteq \mathbb{E}$ is a purely inseparable extension. Then the following are true.\label{cw}
\begin{enumerate}
\item If $\mathbb{F}\subseteq \mathbb{K}\subseteq \mathbb{E}$ is a chain of field extensions, then $\mathbb{K}$ is purely inseparable over  $\mathbb{F}$ and $\mathbb{E}$ 
is purely inseparable over $\mathbb{K}$.
\item If $\dim_\F\mathbb{E}<\infty$, then $\dim_\F\mathbb{E}$ is a power of $p$. 
\end{enumerate}
\item Suppose that $\mathbb{F}\subseteq \mathbb{K}\subseteq \mathbb{E}$ 
is a chain of field extensions such that $\mathbb{F}\subseteq \mathbb{K}$ and $\mathbb{K}\subseteq \mathbb{E}$ are both purely inseparable. Then $\mathbb{E}$ is purely inseparable over $\mathbb{F}$. 
\end{enumerate}
\end{corollary}

\subsection{Integrally closed rings}
We review some results on integrally closed rings. Our standard reference
for this section is Matsumura's book~\cite{matsumura86}. For a domain $A$, we let 
$\fracfield{A}$ denote the field of fractions of $A$.
We denote by $\height I$ the height  of an ideal $I$ in a commutative ring.

\begin{lemma}[1.15 Lemma, \cite{ben-shimol}]\label{gauss}
  If $A$ is a domain, then a monic polynomial $f(x)$ is prime as an element of 
  $A[x]$ if and only if it is prime as an element of $\fracfield A[x]$.
\end{lemma}

\begin{lemma}\label{rr}
  Suppose that $A$ is a finitely-generated noetherian domain, let 
  $A\subseteq B$ be an integral extension of domains, and let 
  $\alpha\in B\setminus A$ be a root of a polynomial $f(x)\in A[x]$.
  \begin{enumerate} 
\item  If $f(x)$ is prime in $A[x]$, then $A[\alpha]\cong A[x]/(f(x))$.
\item If $A$ is factorial and  
$f(x)$ is irreducible in $A[x]$, then $A[\alpha]\cong A[x]/(f(x))$.
  \end{enumerate}
  \end{lemma}
  \begin{proof}
(1)  Let us consider the surjective homomorphism $\varphi: A[x]\to A[\alpha]$ defined by  
  $g(x)\mapsto g(\alpha)$ for all $g(x)\in A[x]$. Clearly, $(f(x))\subseteq \ker\varphi$. Since $f(x)$ is prime in $A[x]$, 
  the ideal $(f(x))$ is prime with height one (by Krull's Principal Ideal Theorem) contained in $\ker \varphi$. 
  In particular,
  $\height{\ker\varphi}\geq1$. Also $\ker\varphi$ is a prime ideal since $A[\alpha]$ is a domain and 
  in order to show that $\ker\varphi=(f(x))$ it suffices to show that 
  $\height{\ker\varphi}=1$. Since $A$ is finitely-generated and noetherian, its Krull dimension $\dim A$ is finite, and so, using standard commutative algebra, we can calculate that
  \begin{align*}
  \height{\ker\varphi}+\dim A&=\height{\ker\varphi}+\dim A[\alpha]= 
  \height{\ker \varphi}+\dim A[x]/\ker(\varphi) \\&\leq \dim A[x]=\dim A+1.
  \end{align*}
  This means that $\height{\ker \varphi}=1$, and so $\ker\varphi=(f(x))$ and also 
  $A[\alpha]\cong A[x]/(f(x))$.

  (2) This follows from statement~(1) since if $A$ is factorial, then 
  so is $A[x]$ and the irreducible element $f(x)\in A[x]$ is prime.
  \end{proof}

If $R\subseteq S$ is an extension of commutative rings, then an element $a\in S$ is said 
to be {\em integral} over $R$ (or $R$-integral) if it is a root of a nonzero monic polynomial in $R[x]$. An integral domain $R$ is said to be {\em integrally closed} 
(or {\em normal} by some authors such as Kemper~\cite{Kemper09}) if the $R$-integral elements of the fraction field 
$\fracfield R$ are precisely the elements of $R$.   

\begin{lemma}\label{corq}
If $A$ is an integrally closed domain, $p$ is a prime, and $\alpha\in A\setminus A^p$, then $x^{p}-\alpha$ 
is a prime polynomial in $A[x]$. 
  \end{lemma}
  \begin{proof}
  Let $\K=\fracfield A$ be the field of fractions of $A$ and set 
  $\K^p=\{a^p\mid a\in\K\}$.
  Let $\alpha$ be as in the lemma. We claim that $\alpha\not\in \K^{p}$.
  Assume that $\alpha\in \K^{p}$, and hence there exists $u\in \K$ such that $\alpha=u^{p}$. Thus $g(t)=t^{p}-\alpha$ is an integral equation for $u$ over $A$, and, as $A$ is integrally closed, this implies that $u\in A$, and so $\alpha=u^p\in A^{p}$, 
  which is a contradiction to the conditions of the lemma. Hence the claim 
  that $\alpha\not\in \K^p$ is valid. It follows 
  from~\cite[Theorem~9.1]{Lang02} that $x^{p}-\alpha$ is irreducible in $\K[x]$,
  and it follows also that $x^p-\alpha$ is prime in $\K[x]$. 
  Now Lemma~\ref{gauss} implies that the same polynomial is prime in $A[x]$. 
  \end{proof}

We rely on Serre's criterion to verify if a commutative ring is integrally closed. 
  Let $R$ be a commutative noetherian ring and let $k\geq0$ an integer. 
  The depth of an ideal $I\subseteq R$ is denoted by $\mbox{dt}(I)$. 
  \begin{enumerate}
  \item[$(\mbox{R}_{k})$] We say that $R$ satisfies Serre's condition $(\mbox{R}_{k})$
  if for all $P\in\spec R$ with $\height P\leq k$, the local ring $R_{P}$ 
  is regular. 
  \item[$(\mbox{S}_{k})$] We say that $R$ satisfies Serre's condition $(\mbox{S}_{k})$ 
  if $\mbox{dt}(R_{P})\geq \min\left\{k, \height P\right\}$ holds for all 
  $P\in\spec R$.
  \end{enumerate}
  
  Note that a noetherian ring is Cohen-Macaulay if and only if it satisfies Serre's condition $(\mbox{S}_i)$ for all $i\geq 0$. 

  \begin{theorem}[Serre's Criterion for Normality; Theorem~23.8~\cite{matsumura86}]
\label{serrecrit}    The following is true for a  noetherian domain $R$. 
    \begin{enumerate} 
      \item If $R$ satisfies Serre's condition $(\mbox{R}_{1})$ and $(\mbox{S}_{2})$, then $R$ is integrally closed.\label{Te13}
      \item If $R$ is Cohen-Macaulay, and satisfies Serre's condition $(\mbox{R}_{1})$
      then $R$ is integrally closed.
    \end{enumerate}
  \end{theorem}
  
  The following technical result will be useful to verify that certain rings are 
integrally closed. For a ring $R$ and for $s\geq 1$, let $R[\bft_s]$ denote the 
polynomial ring $R[t_1,\ldots,t_s]$.

\begin{proposition}\label{nj}
  Suppose that $A$ is a regular domain, let $I\subseteq A[\bft_s]$ be a principal ideal
  generated by a prime polynomial $f\in A[\bft_s]$ and set 
  \[
    J=\left(f, d(f) \ | \ d\in \der{A[\bft_s]}\right).
   \]
  If $\mbox{ht}(J)\geq 3$, then  the quotient $A[\bft_s]/I$ is an integrally closed 
  domain.
  \end{proposition}
  \begin{proof}
    Set $R=A[\bft_s]/I$. 
  The polynomial ring $A[\bft_s]$ is regular by~\cite[Theorem~19.5]{matsumura86} 
  and so it is also Cohen-Macaulay by~\cite[Theorem 17.8]{matsumura86}, which, 
  in turn implies that the quotient $R$ 
  is Cohen-Macaulay (see~\cite[Exercise~17.4]{matsumura86}). 
  Thus, by Theorem~\ref{serrecrit}, it suffices to show that $R$ satisfies
  Serre's condition $(\mbox{R}_1)$. 
  Let $\bar{P}\in \spec R$ with $\height{\bar{P}}\leq1$. We are required to show that 
  $(A[\bft_s]/I)_{\bar{P}}$ is regular.
  Note that $\bar{P}=P/I$ for some $P\in \spec{A[\bft_s]}$ with $I\subseteq P$. 
  Moreover, since the Cohen-Macaulay ring $A[\bft_s]$ is a catenary (by~\cite[Theorem~17.9]{matsumura86}) and $0\subset I\subset P$ is a saturated chain of prime ideals,  $\height P\leq2$. If $d(f)\in P$ were true for for all $d\in \der{A[\bft_s]}$, then 
  $P\supseteq J$ would be true, and so $\height P$ would be at least $3$, which is not the case. Therefore, there must exist $d\in \der{A[\bft_s]}$ such that $d(f)\not\in P$. 
  Now the Jacobian Criterion for Regularity (\cite[Theorem~30.4]{matsumura86})
  implies that $A[\bft_s]_{P}/(IA[\bft_s]_{P})\cong (A[\bft_s]/I)_{\bar{P}}$ is regular, 
  and hence $R$ satisfies Serre's condition $(\mbox{R}_{1})$. Now Theorem~\ref{serrecrit} 
  implies that $R$ is integrally closed.
  \end{proof}

\section{The structure of the universal enveloping algebra: the center and the 
$p$-center}\label{sec:center}

Let us introduce the following notation for a Lie algebra $\g$:
\begin{description}
  \item[$C(\g)$] the center of $\g$;
  \item[$U(\g)$] the universal enveloping algebra of $\g$;
  \item[$Z(\g)$] the center of  $U(\g)$;
  \item[$D(\g)$] the division algebra of fractions of $U(\g)$;
  \item[$K(\g)$] the field of fractions of $Z(\g)$;
  \item[$Z_p(\g)$] the $p$-center of  $U(\g)$;
  \item[$K_p(\g)$] fraction field of the $p$-center $Z_p(\g)$.
\end{description} 
%
%


The {\em $p$-center} $Z_p(\g)$ of $U(\g)$ for a Lie algebra $\g$ over a field of characteristic $p$ 
is defined, for instance, in~\cite[Section~5.1]{strade98} where it is denoted $O(\g)$, 
or in~\cite[\S 2]{zassenhaus53}, where it is denoted $\mathfrak o$.  Here we only define it for Lie algebras over a  field $\F$ of characteristic $p$
that are nilpotent of class less than or equal to $p$ (that is, $\g^{p+1}=0$). 
Such a nilpotent Lie algebra $\g$ is restrictable by the trivial $p$-map $x^{[p]}=0$ 
for all $x\in \g$, and hence by~\cite[Theorem~1.3]{strade98}, the $p$-center $Z_p(\g)$ of 
such a Lie algebra can easily be described as follows. Suppose that 
$\{x_1,\ldots,x_k\}$ is a basis for $\g$ such that $\{x_{d+1},\ldots,x_k\}$ is a 
basis for $C(\g)$; then 
\begin{equation}\label{pcenter}
  Z_p(\g)=\F[x_1^p,\ldots,x_d^p,x_{d+1},\ldots,x_k]
\end{equation}
(considered as a subalgebra in $U(\g)$).
Since $\ad\,{x_i^p}=(\ad\,{x_i})^p=0$ for all $i\in\{1,\ldots,d\}$, we obtain that 
$Z_p(\g)\subseteq Z(\g)$.  Furthermore, $U(\g)$ is a free $Z_p(\g)$-module 
of rank $p^d$. 

The division algebra $D(\g)$ is constructed following Ore's construction; 
see~\cite[3.6.13]{dixmier96}. As $Z(\g)$ and $Z_p(\g)$ are  integral domains, their 
fields of fractions $K(\g)$ and $K_p(\g)$ can be embedded into $D(\g)$.  
In characteristic $p$, 
Zassenhaus~\cite[\S 2]{zassenhaus53} observed that the division algebra 
$D(\g)$ can be obtained by adjoining only quotients with denominators in 
$Z_p(\g)$, and this implies that $\dim_{K_p(\g)}D(\g)=p^d$ where, as above,
 $d=\dim g-\dim C(\g)$. 
 
  \begin{theorem}[Corollary~4.7.2~\cite{dixmier96}]\label{lk}
  If $\g$ is a nilpotent Lie algebra, 
  then $K(\g)$ coincides with the center of $D(\g)$.
  \end{theorem}

  The following result describes the relationship between  
the various rings and fields introduced so far in this section.

  \begin{theorem}[Theorem~1.3, page~203, \cite{strade98}]\label{lm}
    Let $\g$ be a finite-dimensional nilpotent Lie algebra 
    over a field $\F$ of 
    prime characteristic $p$ such that the nilpotency class of $\g$ is less than 
    or equal to $p$. Then the following hold.
    \begin{enumerate}
      \item $Z_p(\g)$ is isomorphic to a polynomial algebra over $\F$ in $\dim \g$ variables. In particular, $Z_p(\g)$ is integrally closed. 
     \item  $U(\g)$ is a free $Z_p(\g)$-module of rank 
      $p^{\dim \g-\dim C(\g)}$.  
      \item $\dim_{K_p(\g)}D(\g)=p^{\dim \g-\dim C(\g)}$.
      \item $\dim_{K(\g)}D(\g)=p^{2m}$ for some $m\geq 1$. 
        \item $Z(\g)$ is an integral ring extension of $Z_p(\g)$  
        while $K(\g)$ is a finite purely inseparable field extension of $K_p(\g)$.  
      \item If $K(\g)=K_p(\g)$, then $Z(\g)=Z_p(\g)$. 
    \end{enumerate}
  \end{theorem}
\begin{proof}
  Items~(1) and (2) follow from~\cite[Theorem~1.3, Section~5.1]{strade98}.
  Item~(3) follows from the fact that $D(\g)$ can be obtained by adjoining quotients
  with denominators in $Z_p(\g)$ only (see~\cite[\S 2]{zassenhaus53}). 
  Statement~(4) follows from 
  Theorem~\ref{lk} and from the fact that the dimension of a division algebra over its 
  center is a square number.  Statement~(5) follows from the fact that 
  $Z(\g)$ is a $Z_p(\g)$-submodule in the finitely generated $Z_p(\g)$-module $U(\g)$, 
  and hence it is finitely generated.  Finally, item (6) follows from that fact that $Z_p(\g)$ is integrally closed.
\end{proof}

  \begin{proposition}\label{yhb}
  Let $\g$ be a finite-dimensional nilpotent Lie algebra over a field $\mathbb{F}$ 
  of prime characteristic $p$ such that the  nilpotency class of $\g$ is less than 
  or equal to  $p$. Assume that there exist $z_{1},\ldots, z_{s}\in Z(\g)$  such that 
  $z_1^p,\ldots,z_s^p\in Z_p(\g)$ and  
  \begin{equation}\label{tower1}
  K_{p}(\g)\subset K_{p}(\g)(z_{1})\subset K_{p}(\g)(z_{1}, z_{2})\subset\cdots\subset K_{p}(\g)(z_{1},\ldots, z_{s}).
  \end{equation}
Then  $\dim_{K_{p}(\g)} K_{p}(\g)(z_{1}, \ldots, z_{s})=p^{s}$. 
Furthermore, if $s=\dim \g-\dim C(\g)-2$, then $K_{p}(\g)(z_{1}, \ldots, z_{s})=K(\g)$ and 
$\dim_{K(\g)}D(\g)=p^2$. 
  \end{proposition}
  \begin{proof}
  It suffices to show that $d_i=\dim_{K_{p}(\g)(z_{1}, \ldots, z_{i-1})}K_{p}(\g)(z_{1}, \ldots, z_{i-1},z_i)=p$ for all $i$. 
  Since the inclusions in~\eqref{tower1} are proper, 
  $z_{i}\not\in K_{p}(z_{1},\ldots,z_{i-1})$, but $z_{i}^{p}\in K_{p}(\g)$, 
  and hence
  $d_i\leq p$. On the other hand $K_{p}(\g)(z_{1}, \ldots, z_{i-1},z_i)$ is purely inseparable over $K_{p}(\g)(z_{1}, \ldots, z_{i-1})$ and so $d_i\geq p$. 
  Thus $d_i=p$ as is required.

Let us show the second statement. By Theorem~\ref{lm},  
\[
\dim_{K_p(\g)} D(\g) = p^{\dim\g-\dim C(\g)},
\] 
while Theorem~\ref{lk} implies that 
$\dim_{K(\g)}D(\g)\geq p^2$. Thus if a chain as in~\eqref{tower1} exists
with $s=\dim\g-\dim C(\g)-2$, then it must follow that $K_p(\g)(z_1,\ldots,z_s)=K(\g)$ and 
$\dim_{K(\g)}D(\g)=p^2$. 
\end{proof}

\begin{lemma}\label{dimensao}
  Let $\mathfrak{g}$ be a finite-dimensional nilpotent Lie algebra 
  over a field $\mathbb{F}$ of prime characteristic $p$ such that 
  the nilpotency class of $\g$ is less than or equal to $p$.  Suppose that there exist $x, y, z\in \mathfrak{g}$ 
  such that $x\not\in C(\mathfrak{g})$,  
  $[x, y]=0$, $[x,z]=0$, but $[y,z]\neq 0$.  
  Then $\dim_{K(\mathfrak{g})}D(\mathfrak{g})\geq p^{4}$. 
  \end{lemma}
  \begin{proof}
  Suppose that $x, y\in \mathfrak{g}$ are as in the lemma. 
  Set $D_{1}=K(\mathfrak{g})[x]$ and $D_{2}=K(\mathfrak{g})[x,y]$.
  Since $[x, y]=0$, $D_{1}$ and $D_{2}$ are commutative subalgebras 
  of $D(\g)$. By Theorem~\ref{lm}, $D_{1}$ and $D_{2}$ are finite-dimensional 
  integral domains over the field $K(\mathfrak{g})$ and so they both are fields
  (see \cite[Corollary~1.2.3]{finitealgebra}). Moreover, $D_2\neq D(\g)$ since 
  $D(\g)$ is not commutative. Since $[x,z]=0$, $D_1\leq C_{D(\g)}(z)$, 
  and, as $[y,z]\neq 0$, it follows that $y\not\in C_{D(\g)}(z)$ and hence 
  $y\not\in K(\g)[x]$, and, in particular,
  $D_{1}\neq D_{2}$. The condition 
  $\mbox{cl}(\mathfrak{g})\leq p$ implies that 
  $x^{p}, y^{p}\in K(\mathfrak{g})$ which gives that 
  $\dim_{K(\mathfrak{g})}D_{1}\leq p$ and $\dim_{D_{1}}D_{2}\leq p$. 
  By Corollary~\ref{cor:insep}, $K(\mathfrak{g})\subseteq D_{1}$ and $D_{1}\subseteq D_{2}$ are purely inseparable extensions, and so $\dim_{K(\mathfrak{g})}D_{1}=\dim_{D_{1}}D_{2}=p$. Hence, $\dim_{K(\mathfrak{g})}D_{2}=p^{2}$ and  $\dim_{K(\mathfrak{g})}D(\mathfrak{g})\geq p^{4}$ (considering that $\dim_{K(\g)}D(g)$ is a square number).
  \end{proof}

\section{The generators of $Z(\g)$}\label{sec:th1}

In this section we prove the following theorem.

\begin{theorem}\label{th1}
  Let $\g$ be a nilpotent Lie algebra of dimension at most~$6$ over a field 
  $\F$ of prime characteristic $p$ and suppose that $p$ is not smaller than the 
  nilpotency class of~$\g$. 
  \begin{enumerate}
    \item If $\g$ does not appear in Table~\ref{table:Zgpgen}, then $Z(\g)=Z_p(\g)$.
    \item If $\g$ appears in Table~\ref{table:Zgpgen}, then 
     $Z(\g)$ is equal to the integral closure of $Z_p(\g)[z_1,\ldots,z_k]$ in its field of fractions
    where $z_1,\ldots,z_k$ are the elements 
    in the $z_i$-column of the corresponding row of Table~\ref{table:Zgpgen}.
    \item If $\g$ is as in item~(2) and is not isomorphic to one of the algebras $\g_{5,5}$, $\g_{6,18}$, 
    or $\g_{6,25}$, then 
    \[
      Z(\g)=Z_p(\g)[z_1,\ldots,z_k].
    \]
  \end{enumerate}
\end{theorem}

We split the proof into several subsections grouping together algebras that show
similar behaviour. The arguments in this section are quite computational. We give enough detail for the reader to follow the arguments; more detailed calculations are available in the first author's PhD thesis~\cite{vanderlei}.

\subsection{$\boldsymbol{Z(\g)=Z_p(\g)}$} 
Suppose that $\g$ is one of the algebras 
$\g_3$, $\g_{5,1}$, $\g_{5,3}$, $\g_{5,6}$, $\g_{6,7}^{(2)}(\varepsilon)$ (in characteristic $2$ only),  $\g_{6,22}(\varepsilon)$, $\g_{6,23}$, 
$\g_{6,24}(\varepsilon)$, $\g_{6,27}$, $\g_{6,28}$ with some $\varepsilon\in\F$. 
We claim that $Z(\g)=Z_p(\g)$.

Let us temporary suppose that $\varepsilon\in\{0,1\}$ in case $\g=\g_{6,22}(\varepsilon)$ or $\g=\g_{6,24}(\varepsilon)$ 
and $\varepsilon=0$ when $\g=\g_{6,7}^{(2)}(\varepsilon)$.
First we show that $K(\g)=K_p(\g)$. The 
proof being the same for all these algebras, we present enough details so that 
the reader can reconstruct the argument in each individual case, but we omit some 
calculations. By Theorem~\ref{lm}, $\dim_{K_p(\g)}D(\g)=p^{\dim\g-\dim C(\g)}$, while 
$\dim_{K(\g)}D(\g)=p^{2m}$ where $m$ is a positive integer. Thus 
\begin{align*}
  p^{\dim\g-\dim C(\g)}&=\dim_{K_p(\g)}D(\g)=\dim_{K(\g)}D(\g)\cdot\dim_{K_p(\g)} K(\g)\\&=
  p^{2m}\dim_{K_p(\g)} K(\g), 
\end{align*}
and so 
\[
  \dim_{K_p(\g)} K(\g)=p^{\dim\g-\dim C(\g)-2m}.
\]
We claim that $\dim\g-\dim C(\g)-2m=0$ which implies for these algebras 
that $K_p(\g)=K(\g)$. This is clear if $\g=\g_3$, since $m\geq 1$ and 
$\dim\g_3-\dim C(\g_3)=2$. Assume in the rest of the argument that $\g$ is one of the 
other algebras.
Inspection of Tables~\ref{table:dim5}--\ref{table:dim6} shows in all cases that 
$\dim\g-\dim C(\g)\leq 4$. Hence, we are only required to show that $m\geq 2$, that is,  
$\dim_{K(\g)}D(\g)\geq p^4$. However, in each case, one can choose distinct $x,y,z\in\g$ 
such that $x\not\in C(\g)$, $[x,y]=[x,z]=0$ and
$[y,z]\neq 0$. 
Thus Lemma~\ref{dimensao} implies that $\dim_{K(\g)}D(\g)\geq p^4$ which gives
that $K(\g)=K_p(\g)$, as required. 

Now the facts that $Z_p(\g)$, being isomorphic
to a polynomial algebra, is integrally closed, and $Z(\g)$ is an integral extension of 
$Z_p(\g)$ with the same fraction field imply that $Z(\g)=Z_p(\g)$. 

It remains to consider the case when $\g=\g_{6,22}(\varepsilon)$ or 
$\g=\g_{6,24}(\varepsilon)$ such that $\varepsilon\not\in\F\setminus\{0,1\}$
or $\g=\g_{6,7}^{(2)}(\varepsilon)$ with $\varepsilon\neq 0.$
Let $m\in\{22,24\}$. 
If $\varepsilon$ is a square in $\F$ (that is, $\varepsilon\in\F^2$), then $\g_{6,m}(\varepsilon)\cong 
\g_{6,m}(1)$ by~\cite[Theorem~3.1]{cgs12}. Suppose that $\varepsilon\not\in\F^2$ and  
set $\E=\F[t]/(t^2-\varepsilon)$. Then $\E$ is a quadratic field extension  of $\F$ 
such that $\varepsilon\in\E^2$. 
By~\cite[Theorem~Theorem~3.1]{cgs12}, $\g_{6,m}(\varepsilon)\otimes\E\cong 
\g_{6,m}(1)\otimes\E$, and hence  
\[
  Z(\g_{6,m}(\varepsilon)\otimes \E)=Z_p(\g_{6,m}(\varepsilon)\otimes \E)
\]
and 
\[
  Z(\g_{6,m}(\varepsilon))=Z(\g_{6,m}(\varepsilon)\otimes \E)\cap U(\g)=Z_p(\g_{6,m}(\varepsilon)\otimes\E)\cap U(\g_{6,m}(\varepsilon))=Z_p(\g_{6,m}(\varepsilon)).
\]

The same argument can be used for the algebra $\g_{6,7}^{(2)}(\varepsilon)$ over fields of characteristic~2. If $\varepsilon\in\{x^2+x\mid x\in\F\}$, then 
$\g_{6,7}^{(2)}(\varepsilon)\cong \g_{6,7}^{(2)}(0)$ (see~\cite[Theorem~3.1]{cgs12} and~\cite{cgscorr}). 
If this is not the case, then $x^2+x+\varepsilon\in\F[x]$ is an 
irreducible polynomial and we apply the argument in the previous paragraph to $\g_{6,7}^{(2)}(\varepsilon)\otimes \E$ where $\E=\F[x]/(x^2+x+\varepsilon)$.

\subsection{$\boldsymbol{Z(\g)}$ is a simple extension of $\boldsymbol{Z_p(\g)}$}
\label{sec:simple_ext}
Here we analyse the algebras $\g_{4}$, $\g_{5,2}$, $\g_{5,4}$, 
$\g_{6,10}$, $\g_{6,11}$, $\g_{6,12}$, $\g_{6,13}$, $\g_{6,14}$, $\g_{6,15}$, 
$\g_{6,16}$, $\g_{6,17}$, $\g_{6,19}(\varepsilon)$, $\g_{6,20}$, 
$\g_{6,21}(\varepsilon)$, and $\g_{6,26}$ with $\varepsilon\in\F^*$. 
Suppose that $\g$ is one of these algebras and let $z$ be the element 
that appears in Table~\ref{table:Zgpgen} 
in the intersection of the row corresponding to $\g$ and the $z_i$-column. 

\begin{lemma}
$z\in Z(\g)\setminus Z_p(\g)$, but $z^p\in Z_p(\g)$. 
 \end{lemma}
\begin{proof}
The first statement follows by elementary, yet cumbersome, calculation. The fact that $z^p\in Z_p(\g)$ follows from the fact that, in each case,
$z$ can be written as a sum 
$a_1+\cdots+a_k$ such that the Lie algebra generated by $\{a_1,\ldots,a_k\}$ inside $U(\g_)$ is nilpotent of class smaller than $p$. Thus 
\[
\left(\sum_{i=1}^k a_i\right)^p=\sum_{i=1}^k a_i^p;
\]
see the discussion, in particular Lemma~1.2, in~\cite[Section~2.1]{strade98}. 
\end{proof}

\begin{lemma}
$Z_p(\g)[z]$ is integrally closed and  $Z(\g)=Z_p(\g)[z]$.
\end{lemma}
\begin{proof}
Set $\alpha=z^p$; then $\alpha\in Z_p(\g)$. By Lemma~\ref{corq}, the polynomial $f=t^p-\alpha$
is prime in $Z_p(\g)[t]$ and hence Lemma~\ref{rr} implies that 
\[
  Z_p(\g)[z]\cong Z_p(\g)[t]/(f).
\]
Note that $Z_p(\g)$ is a polynomial algebra in $\dim\g$ variables and hence 
$Z_p(\g)[t]$ is also a polynomial algebra in $\dim\g+1$ variables 
(one of which is $t$). Letting $s=\dim \g+1$, 
there is an isomorphism $\psi:Z_p(\g)[t]\mapsto \F[\bft_s]$
mapping the generators of the $p$-center into the variables $t_1,\ldots,t_{s-1}$ 
and mapping $t$ to $t_{s}$. 
Thus 
\[
  Z_p(\g)[z]\cong\F[\bft_s]/(\hat f) 
\]
where $\hat f=\psi(f)$. For instance, in the case of 
$\g=\g_4$, we have $z=x_3^2-2x_2x_4$. The elements $x_1^p$, $x_2^p$, $x_3^p$, $x_4$ of $Z_p(\g)$ correspond to the
indeterminates $t_1$, $t_2$, $t_3$, and $t_4$, respectively, and so $\hat f=t_5^p-t_3^2+2t_2t_4^p$. Thus 
\[
  Z_p(\g_4)[z]\cong \F[\bft_5]/(\hat f)=\F[t_1,t_2,t_3,t_4,t_5]/(t_5^p-t_3^2+2t_2t_4^p).
\]
Returning to the general case, let $J$ denote the ideal of $\F[\bft_s]$ defined as
\[
  J=\left(\hat f,d(\hat f)\mid d\in\mbox{Der}(\F[\bft_s])\right).
\]
Considering the derivations $\mbox{Der}(\F[\bft_s])$, it is possible to find 
distinct $i,j\in\{1,\ldots,s\}$ such that $t_i,t_j\in J$, 
\[
  0\subset (t_i)\subset (t_i,t_j)
\]
and $\hat f\not \in (t_i,t_j)$. Therefore if $P$ is a prime ideal containing $J$, then 
\[
  0\subset (t_i)\subset( t_i,t_j)\subset P
\]
is a strictly increasing chain of prime ideals which shows that $\height P\geq 3$,
and hence $\height{J}\geq 3$. Now Proposition~\ref{nj} gives that 
$\F[\bft_s]/(\hat f)$ is integrally closed and so $Z_p(\g)[z]$ is 
also integrally closed.

Let us now verify that $K_p(\g)(z)=K(\g)$. In the case of $\g=\g_4$, $\g=\g_{5,2}$, $\g=\g_{5,4}$, or $\g=\g_{6,26}$, 
$\dim\g-\dim C(\g)=3$, and so $K(\g)=K_p(\g)(z)$ holds by Proposition~\ref{yhb}. Let us 
consider the other cases simultaneously. By Theorem~\ref{lm}, 
\begin{align*}
  p^{\dim\g-\dim C(\g)}&=\dim_{K_p(\g)}D(\g)\\&=\dim_{K_p(\g)} K_p(\g)(z)\cdot 
  \dim_{K_p(\g)(z)} K(\g)\cdot \dim_{K(\g)}D(\g)\\&=p^{2m+1}\dim_{K_p(\g)(z)} K(\g)
\end{align*}
for some $m\geq 1$. Hence 
\[
  \dim_{K_p(\g)(z)}K(\g)=p^{\dim\g-\dim C(\g)-2m-1}.
\]
Since $\g$ is nilpotent and $\dim\g\leq 6$, we have 
 $\dim\g-\dim C(\g)\leq 5$, and hence $K_p(\g)(z)=K(\g)$ follows once we show that 
 $m\geq 2$. This, however, follows from Lemma~\ref{dimensao}, since in all cases, 
 there are $x\not\in C(\g)$ such that $[x,y]=[x,z]=0$, 
 but $[y,z]\neq 0$. 

 We have thus shown that $Z(\g)$ is an integral extension of the integrally
 closed ring $Z_p(\g)[z]$ and that the fraction fields of $Z_p(\g)[z]$ and 
 $Z(\g)$ coincide. This implies that $Z(\g)=Z_p(\g)[z]$. 
\end{proof}

\subsection{The algebras $\g_{5,5}$, $\g_{6,18}$, $\g_{6,25}$}\label{sec:threealgs}

Suppose that $\g$ is one of these algebras and let $z_1,\ldots,z_k$ be the 
elements in the intersection of the $z_i$-column and 
the corresponding row of Table~\ref{table:Zgpgen} (in particular, $k=2$ or 
$k=3$). 
It is a simple calculation to show that $z_i\in Z(\g)$ holds for all $i$. 
Set $u=x_5$ if 
$\g=\g_{5,5}$ or $\g=\g_{6,25}$ and set $u=x_6$ for $\g=\g_{6,18}$. Also 
set $n=\dim\g$.  For a domain $A$ and $a\in A$, we denote by 
$A_a$ the localization of $A$ by the monoid generated by $a$. 

The center $Z(\g)$ is described by the following theorem.

  \begin{theorem}\label{th:closures}
    The integral closures of $Z_p(\g)[z_1,\ldots,z_k]$ in 
    $\fracfield{Z_p(\g)[z_1,\ldots,z_k]}$ and in the localization 
    $Z_p(\g)_{u}[z_1,\ldots,z_k]$ coincide 
    and this integral closure is equal to $Z(\g)$. 
  \end{theorem}

Since $z^p$ is an expression of commuting variables, $z_i^p\in Z_p(\g)$; set $\alpha_i=z_i^p$ for all 
$i\in\{1,\ldots,k\}$. Consider the polynomial ring 
$Z_p(\g)[t_1,\ldots,t_k]$ and set $f_i=t_i^p-\alpha_i$ 
for all $i\in\{1,\ldots,k\}$. The algebra $Z_p(\g)$
is isomorphic to the polynomial algebra $\F[\bft_n]$ and so 
$Z_p(\g)[t_1,\ldots,t_k]\cong \F[\bft_{n+k}]$. Let $f\mapsto \hat f$ denote 
the obvious isomorphism between $Z_p(\g)[t_1,\ldots,t_k]$ and $\F[\bft_{n+k}]$ (as in Section~\ref{sec:simple_ext}).

First we present a lemma.

\begin{lemma}\label{lem:g55closed}
  We have that $Z_p(\g)[z_1,\ldots,z_k]\cong \F[\bft_{n+k}]/(\hat f_1,\ldots,\hat f_k)$. Furthermore,  the localization $Z_p(\g)_{u}[z_1,\ldots,z_k]$ is a Cohen-Macaulay ring that satisfies Serre's condition $(R_k)$ for all $k\geq 0$. 
  In particular  $Z_p(\g)_{u}[z_1,\ldots,z_k]$ is integrally closed. 
\end{lemma}
\begin{proof}
We present the proof separately for the three algebras that occur in this section.

Consider first $\g=\g_{5,5}$. In this case, $k=2$, 
\[
z_1=x_4^2-2x_3x_5\quad\mbox{and}\quad z_2=3x_2x_5^2-3x_3x_4x_5+x_4^3
\]
while
\[
  \hat f_1=t_6^p-t_4^2+2t_3t_5^p\quad\mbox{and}\quad 
  \hat f_2=t_7^p-3t_2t_5^{2p}+3t_3t_4t_5^p-t_4^3.
\]
First apply Lemma~\ref{corq} to deduce that $f\in Z_p(\g)[t]$ 
(and hence $\hat f_1\in\F[\bft_6]$) is prime, and thus, by Lemma~\ref{rr}, 
\[
  Z_p(\g)[z_1]\cong Z_p(\g)[t_1]/(f_1)\cong \F[\bft_6]/(\hat f_1).
\]
Furthermore
notice that $Z_p(\g)[t_1]/(f_1)$ is isomorphic to a polynomial ring in one variable
over $Z(\g_{4})$ (this is because $\g_4$ is isomorphic to the 
subalgebra generated by $x_1$ and $x_3$ in $\g_{5,5}$). It was established in Section~\ref{sec:simple_ext} that $Z(\g_{4})$ is 
integrally closed, which shows that $Z_p(\g)[z_1]\cong \F[\bft_6]/(\hat f_1)$ is integrally closed.

Note that 
$f_2\in Z_p(\g)[z_1][t_2]$  
is prime by Lemma~\ref{corq}, and so Lemma~\ref{rr} implies 
that 
\begin{align*}
  Z_p(\g)[z_1,z_2]&=Z_p(\g)[z_1][z_2]\cong Z_p(\g)[z_1][t_2]/(f_2)\cong Z_p(\g)[t_1,t_2]/(f_1,f_2)\\&\cong \F[\bft_7]/(\hat f_1,\hat f_2).
\end{align*}
Now 
\[
  Z_p(\g)_{x_5}[z_1,z_2]\cong Z_p(\g)[z_1,z_2]_{x_5}\cong \F[\bft_7]_{t_5}/(\hat f_1,\hat f_2).
\]
Write $\F[\bft_7]_{t_5}=\F[\bft_6]_{t_5}[t_7]$ and note that $\F[\bft_6]_{t_5}$ 
 is integrally closed, since it is a localization of the integrally closed ring $\F[\bft_6]$~\cite[Proposition~8.10]{Kemper09}. 
 We claim that $\F[\bft_7]_{t_5}/(\hat f_{1}, \hat f_{2})$ is Cohen-Macaulay. 
 First, $\F[\bft_7]_{t_5}$ is the localization of a Cohen-Macaulay 
 ring and so it is also Cohen-Macaulay~\cite[page~136]{matsumura86}. 
 Since $\hat f_1\in \F[\bft_7]_{t_5}$ is prime (Lemma~\ref{corq}), 
 $(\hat f_1)$ is a prime ideal and so 
 $\F[\bft_7]_{t_5}/(\hat f_1)$ has no zero-divisors. Further, $\hat f_2\not\in (\hat f_1)$, 
 and hence $\hat f_{1}$ and $\hat f_{2}$ form a regular sequence in 
 $\F[\bft_7]_{t_5}$.
 Thus it follows from~\cite[Exercise~17.4]{matsumura86} that 
 $\F[\bft_7]_{t_5}/\big(\hat f_{1}, \hat f_{2}\big)$ 
is Cohen-Macaulay.
  
Consider the following derivations of $\F[\bft_7]_{t_5}$: 
\[
  d_{1}=\frac{\partial}{\partial t_{2}}\quad\mbox{and}\quad 
  d_{2}=\frac{\partial}{\partial t_{3}}.
\] Then 
\[
  d_{1}(\hat f_{1})=0,\quad d_{1}(\hat f_{2})=-3t_{5}^{2p}, \quad d_{2}(\hat f_{1})=2t_{5}^p, \quad d_{2}(\hat f_{2})=3t_{4}t_{5}^p
\] 
and so the corresponding Jacobian matrix is 
\[
  (d_{i}\hat f_{j})=\begin{pmatrix} 0 & -3t_{5}^{2p} \\ 2t_{5}^p & 3t_{4}t_{5}^p \end{pmatrix}
\] whose determinant is $6t^{3p}_{5}$ which is nonzero since $p>3$. 
  
  Let $\bar{P}\in \spec{\F[\bft_7]_{t_5}/\big(\hat f_{1}, \hat f_{2}\big)}$ with $\height{\bar{P}}\leq k$. Then there exists $P$ in the prime spectrum  $\spec{\F[\bft_7]_{t_5}}$
  with $t_{5}\not\in P$ and $(\hat f_{1}, \hat f_{2})\subseteq P$
  such that $\bar{P}=P/\big(\hat f_{1}, \hat f_{2}\big)$. Since  $t_{5}\not\in P$,
  $\det(d_{i}\hat f_{j})\not\in P$, and so the Jacobian Criterion for Regularity (see~\cite[Theorem~30.4]{matsumura86}) implies that 
   \[
   (\F[\bft_7]_{t_5})_{P}/(\hat f_{1}, \hat f_{2})_{P}\cong (\F[\bft_7]_{t_5}/(\hat f_{1}, \hat f_{2}))_{\bar{P}}\] 
   is a regular ring. Thus $\F[\bft_7]_{t_5}/\big(\hat f_{1}, \hat f_{2}\big)$ 
   satisfies Serre's condition $(\mbox{R}_{k})$ for all $k\geq 0$. Now it follows by Theorem~\ref{serrecrit} that 
    $\F[\bft_7]_{t_5}/(\hat f_{1}, \hat f_{2})$ is integrally closed and so is $Z_p(\g)_{x5}[z_1, z_2]$.

Let us now consider the case when $\g=\g_{6,18}$. Now $k=3$, 
\[
  z_1=x_5^2-2x_4x_6,\quad z_2=x_5^3-3x_4x_5x_6+3x_3x_6^2,\quad\mbox{and}\quad
  z_3=x_4^2-2x_3x_5+2x_2x_6,
\] 
while 
\begin{align*}
  \hat f_1&=t_7^p-t_5^2+2t_4t_6^p;\\
  \hat f_2&=t_8^p-t_5^3+3t_4t_5t_6^p-3t_3t_6^{2p};\\
  \hat f_3&=t_9^p-t_4^2+2t_3t_5-2t_2t_6^p.
\end{align*}
Note that $Z_p(\g)_{x_6}[z_1,z_2]$ 
is isomorphic to a polynomial ring over $Z_p(\g_{5,5})_{x_5}[z_1,z_2]$ and 
we already showed that this ring is 
integrally closed. Furthermore
\[
  Z_p(\g)_{x_6}[z_1,z_2]\cong Z_p(\g)_{x_6}[t_1,t_2]/(f_1,f_2)\cong\F[\bft_9]_{x_6}/(\hat f_1,\hat f_2).
\]
As $f_3=t_3^p-\alpha_3\in Z_p(\g)_{x_6}[z_1,z_2][t_3]$
is a prime polynomial (Lemma~\ref{corq}),  Lemma~\ref{rr} implies
that 
\begin{align*}
  Z_p(\g)_{x_6}[z_1,z_2,z_3]&\cong Z_p(\g)_{x_6}[z_1,z_2][t_3]/(f_3)\cong  Z_p(\g)_{x_6}[t_1,t_2,t_3]/(f_1,f_2,f_3)\\&\cong
  \F[\bft_9]_{x_6}/(\hat f_1,\hat f_2,\hat f_3).
\end{align*}

We conclude as in the previous case that 
$\F[\bft_9]_{t_6}$ is Cohen--Macaulay. Further, the sequence $\hat f_{1}$, $\hat 
f_{2}$, $\hat f_{3}$ 
is regular in $\F[\bft_9]_{t_6}$, as $\hat f_{1}$ is not a zero-divisor in $\F[\bft_9]_{t_6}$, 
$\hat f_{2}\not\in(\hat f_{1})$ having a term that depends only on $t_{8}$, and 
$\hat f_{3}\not\in(\hat f_{1}, \hat f_{2})$ having a term that depends only on $t_{9}$. 
Thus $\F[\bft_9]_{t_6}/\big(\hat f_{1}, \hat f_{2}, \hat f_{3}\big)$ is Cohen-Macaulay. 
Let us show that  
$\F[\bft_9]_{t_6}/\big(\hat f_{1}, \hat f_{2}, \hat f_{3}\big)$ satisfies Serre's condition $(R_{k})$ for all $k\geq0$. 

Consider the following derivations: $D_{1}=\partial/\partial t_{2}$, 
$D_{2}=\partial/\partial t_{3}$ and $D_{3}=\partial/\partial t_{4}$. Then, 
\begin{align*}
D_{1}(\hat f_{1})&=0, &D_{1}(\hat f_{2})&=0, &D_{1}(\hat f_{3})&=-2t_{6}^{p},\\
D_{2}(\hat f_{1})&=0, &D_{2}(\hat f_{2})&=-3t_{6}^{2p}, &D_{2}(\hat f_{3})&=2t_{5},\\
D_{3}(\hat f_{1})&=2t_{6}^{p}, &D_{3}(\hat f_{2})&=3t_{5}t_{6}^{p}, &D_{3}(\hat f_{3})&=-2t_{4}.
\end{align*}
The corresponding Jacobian matrix is 
\[
  (D_{i}\hat f_{j})=\begin{pmatrix} 0 & 0 & -2t_{6}^{p} \\ 0 & -3t_{6}^{2p} & 2t_{5} \\ 2t_{6}^{p} & 3t_{5}t_{6}^{p} & -2t_{4} \end{pmatrix}
  \] 
and $\det(D_{i}\hat f_{j})=-12t^{4p}_{6}\neq0$, since $p\geq5$. 

Let $\bar{P}\in \mbox{Spec}\big(\F[\bft_9]_{t_6}/\big(\hat f_{1}, \hat f_{2}, \hat f_{3}\big)\big)$. Then, $\bar{P}=P/\big(\hat f_{1}, \hat f_{2}, \hat f_{3}\big)$ for some $P\in \mbox{Spec}(\F[\bft_9]_{t_6})$ with $t_{6}\not\in P$ and $(\hat f_{1}, \hat f_{2}, \hat f_{3})\subseteq P$. Therefore, $\det(D_{i}\hat f_{j})\not\in P$, 
as $t_{6}\not\in P$. By the Jacobian Criterion for Regularity, 
we obtain that $(\F[\bft_9]_{t_6})_{P}/\big(\hat f_{1}, \hat f_{2}, \hat f_{3}\big)_{P}$ is a regular ring, that is, $\big(\F[\bft_9]_{t_6}/\big(\hat f_{1}, \hat f_{2}, \hat f_{3}\big)\big)_{\bar{P}}$ is regular. Now it follows that $\F[\bft_9]_{t_6}/\big(\hat f_{1}, \hat f_{2}, \hat f_{3}\big)$ 
satisfies Serre's condition $(R_k)$ for all $k\geq0$. Hence Theorem~\ref{serrecrit} implies that  $\F[\bft_9]_{t_6}/\big(\hat f_{1}, \hat f_{2}, \hat f_{3}\big)$  and $Z_p(\g)_{x_6}[z_1,z_2,z_3]$ are integrally closed.

Consider finally the case when $\g=\g_{6,25}$. In this case, $k=2$, 
\[
  z_1=x_3^2-2x_2x_5\quad\mbox{and}\quad z_2=x_3x_6-x_4x_5,
\]
while 
\[
  \hat f_1=t_7^p-t_3^2+2t_2t_5\quad\mbox{and}\quad \hat f_2=
  t_8^p-t_3t_6^p+t_4t_5.
\]
We can use the same argument that we used for $\g_{5,5}$ to show
that 
\[
  Z_p(\g)[z_1,z_2]\cong Z_p(\g)[t_1,t_2]/(f_1,f_2)\cong\F[\bft_8]/(\hat f_1,\hat f_2)
\]
and that $\F[\bft_8]_{t_5}/(\hat f_1,\hat f_2)$ is Cohen--Macaulay. 

Let us show that 
$\F[\bft_8]_{t_5}/\big(\hat f_{1}, \hat f_{2}\big)$ satisfies Serre's condition 
$(R_{k})$ for all $k\geq0$.
Consider $D_{1}=\partial/\partial t_{2}$ and  $D_{2}=\partial/\partial t_{4}$. 
Then, 
\[
D_{1}(\hat f_{1})=2t_{5}, \quad D_{1}(\hat f_{2})=0, \quad D_{2}(\hat f_{1})=0 \quad \text{and} \quad D_{2}(\hat f_{2})=t_{5}.
\] 
The corresponding Jacobian matrix is 
\[
  (D_{i}\hat f_{j})=\begin{pmatrix} 2t_{5} & 0 \\ 0 & t_{5} \end{pmatrix}
\] 
whose determinant is non-zero, since $p\geq3$. Now the usual argument 
used for the previous algebras in this proof implies that 
%
$\F[\bft_8]_{t_5}/\big(\hat f_{1}, \hat f_{2}\big)$ satisfies Serre's condition $(R_{k})$ for all $k\geq0$. Hence 
$\F[\bft_8]_{t_5}/\big(\hat f_{1}, \hat f_{2}\big)$ and $Z_p(\g)_{x_5}[z_1,z_2]$ are integrally closed.
\end{proof}

\begin{proof}[The proof of Theorem~\ref{th:closures}]
  First note that Proposition~\ref{yhb} implies that  $K(\g)=K_p(\g)(z_1,\ldots,z_k)$ and so the fraction fields of $Z_p(\g)[z_1,\ldots,z_k]$, 
  $Z_p(\g)_{u}[z_1,\ldots,z_k]$ and $Z(\g)$ are all equal. Let us denote this field by $K$.
  For an extension $R\subseteq S$ of rings, let $R^S$ denote the integral closure of $R$ 
  in $S$. Then Lemma~\ref{lem:g55closed} implies that
  \[
    Z_p(\g)[z_1,\ldots,z_k]^K\subseteq  Z_p(\g)_{u}[z_1,\ldots,z_k]^K=Z_p(\g)_{u}[z_1,\ldots,z_k].
  \] 
  Thus the integral closure $Z_p(\g)[z_1,\ldots,z_k]^K$ is contained in 
  $Z_p(\g)_{u}[z_1,\ldots,z_k]$ and must be equal to 
  $Z_p(\g)[z_1,\ldots,z_k]^{Z_p(\g)_{u}[z_1,\ldots,z_k]}$. This proves the first statement. 

The extension $Z_p(\g)[z_1,\ldots,z_k]\subseteq Z(\g)$ is integral (inside $K$), and so the inclusion
$Z(\g)\subseteq  Z_p(\g)[z_1,\ldots,z_k]^K$ holds. Also, $Z_p(\g)[z_1,\ldots,z_k]^K\subseteq Z(\g)^K=Z(\g)$. 
Thus $Z_p(\g)[z_1,\ldots,z_k]^K=Z(\g)$.
\end{proof}

By Theorem~\ref{th:closures}, finding generators for $Z(\g_{5,5})$ requires finding 
generators for the integral closure of 
\[
  Z_p(\g_{5,5})[z_1,z_2]\cong \mathbb{F}[x_2,x_3,x_4,x_5,t_1,t_2]/\left(f_1,f_2\right).
\]
where
\begin{align*}
f_1&=t^{p}_{1}-x_{4}^{2}+2x_{3}x_{5}^p\\
f_2&=t_{2}^{p}-3x_{2}x_{5}^{2p}+3x_{3}x_{4}x_{5}^p-x_{4}^{3}
\end{align*}
The same theorem shows that there are generators  of the form $g/x_5^\ell$ where 
$g$ is an element of the ring $\F[x_1^p,x_2^p,x_3^p,x_4^p,x_5,z_1,z_2]$ and $\ell\geq 0$. 

The problem of explicitly calculating the integral closure of a commutative domain has
 been considered by several authors in computational commutative algebra; see for 
 instance~\cite{Vas}. In our context, the algorithm that seems most relevant is the one by 
 Singh and Swanson~\cite{SS} which is implemented in the computational algebra 
 system Macaulay2~\cite{M2}. We used the Macaulay2 implementation of the algorithm by 
 Singh and Swanson to calculate explicit generators for the integral closure of $Z_p(\g_{5,5})[z_1,z_2]$ in the cases when $\F=\F_5$ and $\F=\F_7$. 

Suppose first that $\F=\F_5$. The function {\tt integralClosure} of Macaulay2 gives that 
the integral closure of $Z_p(\g_{5,5})[z_1,z_2]$ (that is, $Z(\g_{5,5})$, see Section~\ref{sec:threealgs}) is generated by the following elements  as a $Z_p(\g_{5,5})[z_1,z_2]$-module:
\begin{align*}
  &\frac{{z_1}\,{z_2}+{x_4^5}}{{x_5}^{2}},\quad \frac{{z_1}^{4}-{z_1}\,{z_2}^{2}-2\,{x_4^5}\,{z_2}}{{x_5}^{4}},\quad \frac{{z_1}^{3}{z_2}-{x_4^5}\,{z_1}^{2}+2\,{z_2}^{3}}{{x_5}^{4}},\quad \frac{{z_1}^{3}+{z_2}^{2}}{{x_5}^{2}}\\
  & \frac{2\,{x_2^5}\,{x_5}^{5}{z_1}^{2}+{x_3^5}\,\
  {z_1}^{3}{z_2}-{x_3^5}\,{x_4^5}\,{z_1}^{2}+2\,{
  x_3^5}\,{z_2}^{3}}{{x_5}^{2}{z_1}\,{z_2}+{x_4^5
  }\,{x_5}^{2}},\quad 
  \frac{2\,{x_2^5}\,{x_5}^{4}{z_1}\,{z_2}-{x_3^5
      }^{2}{x_5}^{4}-{x_2^5}\,{x_4^5}\,{x_5}^{4}}{{z_1
      }^{4}-{z_1}\,{z_2}^{2}-2\,{x_4^5}\,{z_2}},\\
      &\frac{{z_1}^{4}{z_2}+2\,{x_4^5}\,{z_1}^{3}-2\,{z_1}\,{z_2}^{3}-{x_4^5}\,{z_2}^{2}}{{x_5}^{5}}.      
\end{align*}

If $\F=\F_7$, then we obtain the following generating set for $Z(\g_{5,5})$:
{\small
\begin{align*}
  &\frac{{z_1}^{6}-2\,{z_1}^{3}{z_2}^{2}+2\,{z_2}^{4
      }-2\,{x_4^7}\,{z_1}\,{z_2}}{{x_5}^{6}},\quad 
      \frac{{z_1}^{5}{z_2}+2\,{z_1}^{2}{z_2}^{3}+2\,{x_4^7}\,{z_1}^{3}+{x_4^7}\,{z_2}^{2}}{{x_5}^{6}},\\
      &\frac{{z_1}^{3}{z_2}-2\,{z_2}^{3}-3\,{x_4^7}\,{z_1}}{{x_5}^{4}},\quad \frac{{z_1}^{3}+{z_2}^{2}}{{x_5}^{2}},
      \quad \frac{{z_1}^{2}{z_2}-{x_4^7}}{{x_5}^{2}}\\
      &\frac{2\,{x_2^7}\,{x_5}^{7}{z_1}^{2}{z_2}-3\,{x_2^7}\,{x_4^7}\,{x_5}^{7}+2\,{x_3^7}\,{z_1}^{4}{z_2}^{2}-2\,{x_3^7}\,{z_1}\,{z_2}^{4}+2\,{x_3^7
      }\,{x_4^7}\,{z_1}^{2}{z_2}+{x_3^7}\,{x_4^7}^{2}}{{x_5}^{2}{z_1}^{3}{z_2}-2\,{x_5}^{2}{z_2}^{3
      }-3\,{x_4^7}\,{x_5}^{2}{z_1}},\\
      &\frac{2\,{z_1}^{9}+{z_1}^{6}{z_2}^{2}+2\,{z_1}^{3
      }{z_2}^{4}-{x_4^7}\,{z_1}^{4}{z_2}+2\,{z_2}^{6
      }-3\,{x_4^7}\,{z_1}\,{z_2}^{3}-{x_4^7}^{2}{z_1
      }^{2}}{{x_5}^{9}},\\ 
      &\frac{-{x_3^7}\,{z_1}^{7}{z_2}+{x_3^7}\,{z_1}^{
        4}{z_2}^{3}-{x_3^7}\,{x_4^7}\,{z_1}^{5}-2\,{x_3^7
        }\,{z_1}\,{z_2}^{5}-2\,{x_3^7}\,{x_4^7}\,{z_1}^{
        2}{z_2}^{2}-3\,{x_3^7}\,{x_4^7}^{2}{z_2}}{{x_5
        }\,{z_1}^{7}+{x_4^7}^{2}{x_5}},\\
        &\frac{{x_3^7}\,{z_1}^{8}+2\,{x_3^7}\,{z_1}^{5}
      {z_2}^{2}+2\,{x_3^7}\,{z_1}^{2}{z_2}^{4}+{x_3^7}\,{x_4^7}\,{z_1}^{3}{z_2}-{x_3^7}\,{x_4^7}\,{
      z_2}^{3}-{x_3^7}\,{x_4^7}^{2}{z_1}}{{x_5}\,{z_1
      }^{7}+{x_4^7}^{2}{x_5}}\\
      &\frac{2\,{x_2^7}\,{x_5}^{4}{z_1}^{2}{z_2}^{3}+2\,{x_3^7}^{2}{x_5}^{4}{z_1}^{3}+2\,{x_2^7}\,{x_4^7
      }\,{x_5}^{4}{z_1}^{3}+2\,{x_3^7}^{2}{x_5}^{4}
      {z_2}^{2}}{{z_1}^{7}+2\,{z_1}^{4}{z_2}^{2}-2\,{z_1
      }\,{z_2}^{4}+2\,{x_4^7}\,{z_1}^{2}{z_2}+2\,{
      x_4^7}^{2}},\\
      &\frac{-3\,{x_2^7}\,{x_5}^{7}{z_1}^{2}{z_2}-3\,{x_2^7}\,{x_4^7}\,{x_5}^{7}-3\,{x_3^7}\,{z_1}^{7}+{x_3^7}\,{z_1}^{4}{z_2}^{2}-{x_3^7}\,{z_1}\,{z_2}^{4}+{x_3^7}\,{x_4^7}\,{z_1}^{2}{z_2}+{x_3^7}\,{x_4^7}^{2}}{{x_5}^{4}{z_1}^{3}+{x_5}^{4
      }{z_2}^{2}},\\
      &\frac{{z_1}^{5}-{z_1}^{2}{z_2}^{2}+2\,{x_4^7}\,{z_2}}{{x_5}^{4}},\quad 
      \frac{{z_1}^{6}{z_2}-3\,{z_1}^{3}{z_2}^{3}+3\,{x_4^7}\,{z_1}^{4}-{z_2}^{5}-{x_4^7}\,{z_1}\,{z_2}^{2}}{{x_5}^{7}}.
\end{align*}}

\section{Polynomial invariants}\label{sec:pol_inv}

Let $\g$ be a finite-dimensional Lie algebra over a field $\F$. The symmetric
algebra $S(\g)$ is the quotient of the tensor algebra 
\[
  S(\g)=T(\g)/I
\]
where $I$ is the ideal generated by all the elements $xy-yx$ for $x,y\in\g$. The adjoint 
action of $\g$ on  $\g$ can be extended by the Leibniz rule to $S(\g)$, and hence $S(\g)$ can be considered as a $\g$-module. For $x\in\g$ and $f\in S(\g)$, we denote
by $x\cdot f$ the image of $f$ under the action of $x$. The algebra of polynomial invariants $S(\g)^\g$ (also denoted $P(\g)$ of $\g$ is defined 
as 
\[
  P(\g)=S(\g)^\g=\{f\in S(\g)\mid x\cdot f=0\mbox{ for all }x\in\g\}.
\]
Fixing a basis $\{x_1,\ldots,x_n\}$ for $\g$ we obtain that $S(\g)$ is isomorphic to 
the polynomial algebra $\F[x_1,\ldots,x_n]=\F[\bfx_n]$ and the action of $x_i$ on 
$\F[\bfx_n]$ can be expressed as 
$$
  x_i\cdot f=\sum_{j=1}^n[x_i,x_j]\frac{\partial f}{\partial x_j}
  \quad\mbox{for all}\quad f\in\F[\bfx_n]
$$
(see, for example, ~\cite[Chapter 3]{sw}).
Thus the algebra of polynomial invariants $\F[\bfx_n]^\g$ coincides with 
the algebra of polynomials $f\in\F[\bfx_n]$ that satisfy the partial differential equations 
\begin{equation}\label{eq:diff}
  \sum_{j=1}^n[x_i,x_j]\frac{\partial f}{\partial x_j}=0\quad\mbox{for all}\quad
  i\in\{1,\ldots,n\}.
\end{equation}

Suppose that $\F$ has characteristic $p$ and that $\g$ is a nilpotent Lie algebra with nilpotency class less than or equal to $p$. Choosing a basis $\{x_1,\ldots,x_n\}$ such 
that $\{x_{d+1},\ldots,x_n\}$ is a basis for the center $C(\g)$, we obtain that the algebra  
\[
P_p(\g)=\F[x_1^p,\ldots,x_d^p,x_{d+1},\ldots,x_n]
\]
is contained in 
$\F[\bfx_n]^\g$. Furthermore, $P_p(\g)\cong Z_p(\g)$ and 
$P_p(\g)$ plays the same role in $\F[\bfx_n]^\g$ that $Z_p(\g)$ plays
in $Z(\g)$. In particular, as in Theorem~\ref{lm}(5), the invariant algebra 
$\F[\bfx_n]^\g$ is an integral extension of $P_p(\g)$. In fact, $z^p\in P_p(\g)$ for all $z\in F[\bfx n]$.

\begin{theorem}\label{th:int-closed}
  If $\F$ is an arbitrary field, and $\g$ is a finite-dimensional Lie algebra over $\F$ with basis $\{x_1,\ldots,x_n\}$, then the algebra $\F[\bfx_n]^\g$ of polynomial invariants is  integrally closed in its field of fractions.
\end{theorem}
\begin{proof}
  Suppose that $S$ is the integral closure of 
  $\F[\bfx_n]^\g$ in $\mbox{Frac}(\F[\bfx_n]^\g)$. 
  As $\F[\bfx_n]^\g\subseteq \F[\bfx_n]$, $S$ is contained in the 
  integral closure of $\F[\bfx_n]$ which is itself. Thus 
  $S\subseteq \F[\bfx_n]$. 
  
  We claim that  
  \[
    \F[\bfx_n]^\g=    \mbox{Frac}(\F[\bfx_n]^\g)\cap \F[\bfx_n].
  \] The containment
  $\F[\bfx_n]^\g\subseteq  \mbox{Frac}(\F[\bfx_n]^\g)\cap \F[\bfx_n]$ is clear,
  and let us show the other direction that $\mbox{Frac}(\F[\bfx_n]^\g)\cap \F[\bfx_n] 
  \subseteq \F[\bfx_n]^\g$.   The $\g$-action on $\F[\bfx_n]$ can be extended to the 
  field $\K=\F(\bfx_n)$ of rational expressions in $x_1,\ldots,x_n$ using the quotient
  rule: 
  \[
    x\cdot\left(\frac{f}{g}\right)=\frac{(x\cdot f)g-f(x\cdot g)}{g^2}
  \]
  for $x\in\g$  and $f,g\in\F[\bfx_n]$ with $g\neq 0$. 
  Suppose that $f/g\in   \mbox{Frac}(\F[\bfx_n]^\g)\cap \F[\bfx_n]$; that is, 
  $f,g\in \F[\bfx_n]^\g$ and $g\mid f$. If $x\in \g$, then 
  \[
    x\cdot \frac{f}{g}=
    \frac{(x\cdot f)g-f(x\cdot g)}{g^2}=0, 
  \]
  and so $f/g\in \F[\bfx_n]^\g$.

It follows from the claim in the previous paragraph that
  \[
    S\subseteq \mbox{Frac}(\F[\bfx_n]^\g)\cap \F[\bfx_n]=\F[\bfx_n]^\g
  \]
  which implies that $\F[\bfx_n]^\g$ is integrally closed.
\end{proof}

If $\g$ is a finite-dimensional Lie algebra with basis $\{x_1,\ldots,x_n\}$, then 
we denote by $M(\g)$ the ``multiplication table'' for $\g$ with respect to the given 
basis. The $(i,j)$-entry of $M(\g)$ is the product $[x_i,x_j]$ considered as an
element of the field $\F(\bfx_n)$.

\begin{theorem}\label{th:diminv}
  If $\g$ is a nilpotent Lie algebra over a field of characteristic $p$ 
  with nilpotency class not greater than $p$. Then 
  \[
    \dim_{\F(\bfx_n)^\g}\F(\bfx_n)=p^{{\rm rank}\,M(\g)}.
  \]
\end{theorem}
\begin{proof}
  This follows at once from~\cite[Proposition~3.3]{braun}, noting that the conditions
  imply that $\g$ is a restricted Lie algebra and it is also a restricted $\g$-module
  under the adjoint representation.
\end{proof}

\section{The isomorphism between $Z(\g)$ and $P(\g)$}\label{sec:th2}

In this final section we prove the second main theorem of this paper.

\begin{theorem}\label{th2}
  Let $\g$ be a nilpotent Lie algebra of dimension at most~$6$ over a field 
  $\F$ of prime characteristic $p$ and suppose that $p$ is larger than the 
  nilpotency class of~$\g$. Then $Z(\g)\cong P(\g)$. 
\end{theorem}

As in Section~\ref{sec:th1}, some easy calculations are omitted, but they can be found in~\cite{vanderlei}. 

\subsection{Some general remarks}
Suppose that $\{x_{1}, \ldots, x_{n}\}$ is a fixed basis of
$\mathfrak{g}$ such that 
$\{x_{k+1}, \ldots, x_{n}\}$ is a basis of $C(\mathfrak{g})$. 
Set  
\[
  P_p(\g)=\mathbb{F}[x_{1}^{p}, \ldots, x_{k}^{p}, x_{k+1}, \ldots, x_{n}].
\] 
We introduce the following additional notation:
\begin{align*}
\mathbb{F}[\mathbf{x}_{n}]&=\mathbb{F}[x_{1}, \ldots, x_{n}] &\text{and}& &\mathbb{F}(\mathbf{x}_{n})&=\mathbb{F}(x_{1}, \ldots, x_{n}),\\
P(\mathfrak{g})&=\mathbb{F}[x_{1}, \ldots, x_{n}]^{\mathfrak{g}} &\text{and}&  &L(\mathfrak{g})&=\mbox{Frac}\big(\mathbb{F}[x_{1}, \ldots, x_{n}]^{\mathfrak{g}}\big),\\
P_{p}(\mathfrak{g})&=\mathbb{F}[x_{1}^{p}, \ldots, x_{k}^{p}, x_{k+1}, \ldots, x_{n}] &\text{and}& &L_{p}(\mathfrak{g})&=\mathbb{F}(x_{1}^{p}, \ldots, x_{k}^{p}, x_{k+1}, \ldots, x_{n}).
\end{align*}
It follows that $L_{p}(\mathfrak{g})\subseteq \mathbb{F}(\mathbf{x}_{n})$ is
a finite, purely inseparable field extension (Theorem~\ref{Isaa1}) 
of degree $p^{k}$. 
Furthermore, $P_{p}(\mathfrak{g})\subseteq P(\mathfrak{g})$ is 
an integral extension of domains, as  $w^{p}\in P_{p}(\mathfrak{g})$ for all $w\in P(\mathfrak{g})$.

\subsection{Standard filiform algebras}
In the case of the standard filiform Lie algebras (that is, for $\g_{5,5}$ and 
$\g_{6,18}$) the isomorphism between $Z(\g)\cong P(\g)$ is 
observed in~\cite[Example 6.4]{ben-shimol}\label{lkj}. In fact,
the commutative algebra 
$\F[x_1^p,x_2,\ldots,x_n]$ can be naturally embedded into both $\F[\bfx_n]$ and into 
$U(\g)$ and the $\g$-action on $\F[x_1^p,x_2,\ldots,x_n]$ is identical 
whether it is considered as a subalgebra of  
$\F[\bfx_n]$ or $U(\g)$.
Taking $i=2$ in~\eqref{eq:diff}, we obtain that if $f\in P(\g)$, then 
$\partial f/\partial x_1=0$, and so $f\in \F[x_1^p,x_2,\ldots,x_n]$; that is,
$P(\g)\subseteq\F[x_1^p,x_2,\ldots,x_n]$. By Theorem~\ref{th1}, $Z(\g)\subseteq 
\F[x_1^p,x_2,\ldots,x_n]$, 
and hence we may actually write the equality $Z(\g)=P(\g)$.

\subsection{The case when $\boldsymbol{P(\g)=P_p(\g)}$}
Suppose that $\g$ is one of the algebras $\g_3$, $\mathfrak{g}_{5,1}$,
$\mathfrak{g}_{5,3}$, $\mathfrak{g}_{5,6}$, $\g_{6,7}^{(2)}(\varepsilon)$, $\mathfrak{g}_{6,22}(\varepsilon)$, $\mathfrak{g}_{6,23}$, $\mathfrak{g}_{6,24}(\varepsilon)$, $\mathfrak{g}_{6,27}$ or $\mathfrak{g}_{6,28}$. 
Theorem~\ref{th1} implies that $Z(\mathfrak{g})=Z_{p}(\mathfrak{g})$. 
As above, let $\{x_{1}, \ldots, x_{k}, x_{k+1}, \ldots, x_{n}\}$ 
be a basis of $\mathfrak{g}$ such that $\{x_{k+1}, \ldots, x_{n}\}$ 
is a basis for $C(\mathfrak{g})$. 

A polynomial $f\in \mathbb{F}[\mathbf{x}_{n}]$ lies in 
$f\in P(\mathfrak{g})$ if and only if $f$ satisfies the partial 
differential equations in~\eqref{eq:diff}.
Writing down the actual partial differential equations in all cases, one can show that
the system so obtained is equivalent to the system 
\[
  \frac{\partial f}{\partial x_i}=0\quad\mbox{for all}\quad i\in\{1,\ldots,k\} 
\] 
whose set of solutions is $P_p(\g)$. We check one case explicitly here, 
the others are even easier and are left to the reader (see also~\cite{vanderlei}).
Suppose that $\g=\g_{6,22}(\varepsilon)$ with $\varepsilon\in\F$. Then 
the differential equations in~\eqref{eq:diff} have the following form:
\begin{align}
\label{31} x_{5}\frac{\partial f}{\partial x_{2}}+x_{6}\frac{\partial f}{\partial x_{3}}&=0;\\
\label{32}-x_{5}\frac{\partial f}{\partial x_{1}}+\varepsilon x_{6}\frac{\partial f}{\partial x_{4}}&=0;\\
\label{33}-x_{6}\frac{\partial f}{\partial x_{1}}+x_{5}\frac{\partial f}{\partial x_{4}}&=0;\\
\label{34}-\varepsilon x_{6}\frac{\partial f}{\partial x_{2}}-x_{5}\frac{\partial f}{\partial x_{3}}&=0.
\end{align}
We are required to show that $\partial f/\partial x_i=0$ for all $i\in\{1,\ldots,4\}$. 
If $\varepsilon=0$, this follows at once.
If $\varepsilon\neq0$, we multiply the equations~(\ref{31}) and~(\ref{33}) by $x_{5}$ and multiply equations~(\ref{32}) and~(\ref{34}) by $x_{6}$ to obtain
\begin{align}\label{311} x_{5}^{2}\frac{\partial f}{\partial x_{2}}+x_{5}x_{6}\frac{\partial f}{\partial x_{3}}&=0;\\
\label{322}-x_{5}x_{6}\frac{\partial f}{\partial x_{1}}+\varepsilon x_{6}^{2}\frac{\partial f}{\partial x_{4}}&=0;\\
\label{333}-x_{5}x_{6}\frac{\partial f}{\partial x_{1}}+x_{5}^{2}\frac{\partial f}{\partial x_{4}}&=0;\\
\label{344}-\varepsilon x_{6}^{2}\frac{\partial f}{\partial x_{2}}-x_{5}x_{6}\frac{\partial f}{\partial x_{3}}&=0.
\end{align}
Now summing equations~(\ref{311}) and (\ref{344}) and subtracting 
equation~(\ref{322}) from equation~(\ref{333}) we obtain
\begin{equation*}
(x_{5}^{2}-\varepsilon x_{6}^{2})\frac{\partial f}{\partial x_{2}}=0 \quad 
\text{and} \quad (x_{5}^{2}-\varepsilon x_{6}^{2})\frac{\partial f}{\partial x_{4}}=0.
\end{equation*}
This implies that $\partial f/\partial x_{2}=\partial f/\partial x_{4}=0$, and 
then, substituting into equations~(\ref{31}) and~(\ref{32}), we also obtain that  
$\partial f/\partial x_{1}=\partial f/\partial x_{3}=0$.

\subsection{$P(\g)$ is a simple extension of $P_{p}(\mathfrak{g})$}
Suppose in this section that $\g$ is one of the algebras $\g_4$,
$\mathfrak{g}_{5,2}$, $\mathfrak{g}_{5,4}$, $\mathfrak{g}_{6,10}$, 
$\mathfrak{g}_{6,11}$, $\mathfrak{g}_{6,12}$, $\mathfrak{g}_{6,13}$, 
$\mathfrak{g}_{6,14}$, $\mathfrak{g}_{6,15}$, $\mathfrak{g}_{6,16}$, 
$\mathfrak{g}_{6,17}$, $\mathfrak{g}_{6,19}(\varepsilon)$, $\mathfrak{g}_{6,20}$, $\mathfrak{g}_{6,21}(\varepsilon)$ or $\mathfrak{g}_{6,26}$.
In each case, 
\[
  L_{p}(\mathfrak{g})\subseteq L(\mathfrak{g})\subseteq \mathbb{F}(\mathbf{x}_{n})
\]
with $\dim_{L_{p}(\mathfrak{g})}\mathbb{F}(\mathbf{x}_{n})=p^{\dim\mathfrak{g}-\dim C(\mathfrak{g})}$ and, also, by Theorem~\ref{th:diminv},  
$\dim_{L(\mathfrak{g})}\mathbb{F}(\mathbf{x}_{n})=p^{r(\mathfrak{g})}$ with $r(\mathfrak{g})=\mbox{rank}\,M(\g)$.

One can easily verify by inspection for each algebra that can occur for 
$\g$ that $\mbox{rank}\,M(\mathfrak{g})=\dim\mathfrak{g}-\dim C(\mathfrak{g})-1$, and so
\begin{align*}
p^{\dim\mathfrak{g}-\dim C(\mathfrak{g})}&=\dim_{L_{p}(\mathfrak{g})}\mathbb{F}(\mathbf{x}_{n})=\dim_{L(\mathfrak{g})}\mathbb{F}(\mathbf{x}_{n})
\cdot \dim_{L_{p}(\mathfrak{g})}L(\mathfrak{g})\\
&=p^{\dim \mathfrak{g}-\dim C(\mathfrak{g})-1}\cdot \dim_{L_{p}(\mathfrak{g})}L(\mathfrak{g}).
\end{align*}
Hence 
\begin{equation}\label{hj}
\dim_{L_{p}(\mathfrak{g})}L(\mathfrak{g})=p.
\end{equation}
Let $z$ be the element of $\F[\g]$ that occurs in the $z_i$-column of 
Table~\ref{table:Zgpgen} and consider $z$ as an element of $\F[\bfx_n]$. 
Then one may easily check that $z\in P(\g)$, but 
$z\not\in P_{p}(\mathfrak{g})$ and $z^{p}\in  P_{p}(\mathfrak{g})$. Since 
$L_{p}(\mathfrak{g})\subset \mathbb{F}(\mathbf{x}_{n})$ is purely inseparable, 
Corollary~\ref{cor:insep} shows that $L_{p}(\mathfrak{g})\subset L_{p}(\mathfrak{g})(z)$ is purely inseparable and  $\dim_{L_{p}(\mathfrak{g})}L_{p}(\mathfrak{g})(z)=p$. 
Using equation~(\ref{hj}),  we conclude that $\dim_{L_{p}(\mathfrak{g})}L(\mathfrak{g})=p$. Therefore 
$L_{p}(\mathfrak{g})(z)=L(\mathfrak{g})$. 
One may use the same argument that was used to show that 
$Z_p(\g)[z]\cong Z_p(\g)[t]/(t^p-z^p)$ in Section~\ref{sec:simple_ext} 
to prove that $P_{p}(\mathfrak{g})[z]\cong P_p(\g)[t]/(t^p-z^p)$ and so 
$P_p(\mathfrak{g})[z]$ is integrally closed. 
Since $P_{p}(\mathfrak{g})[z]\subseteq P(\mathfrak{g})$ is an integral extension, 
 $P(\mathfrak{g})=P_{p}(\mathfrak{g})[z]$, and also 
$Z(\mathfrak{g})\cong P(\mathfrak{g})$.

\subsection{The algebra $\boldsymbol{\mathfrak{g}_{6,25}}$} 
Let $\mathfrak{g}=\mathfrak{g}_{6,25}$. First notice that 
\[
  M(\g)=\begin{pmatrix} 0 & x_{3} & x_{5} & x_{6} & 0 & 0 \\ -x_{3} & 0 & 0 & 0 & 0 & 0\\ -x_{5} & 0 & 0 & 0 & 0 & 0\\ -x_{6} & 0 & 0 & 0 & 0 & 0\\ 0 & 0 & 0 & 0 & 0 & 0\\ 0 & 0 & 0 & 0 & 0 & 0 \end{pmatrix}
\]
and so $\dim_{L(\g)}\F(\bfx_n)=p^{\mbox{rank}\, M(\g)}=p^2$ (Theorem~\ref{th:diminv}). Furthermore, since $\dim C(\g)=2$, we also have that 
$\dim_{L_{p}(\mathfrak{g})}\mathbb{F}(\mathbf{x}_{n})=p^{\dim \g-\dim C(\g)}=p^4$. Thus we also obtain that 
$\dim_{L_{p}(\mathfrak{g})}L(\mathfrak{g})=p^{2}$.
Set
\[
  z_{1}=x_{3}^{2}-2x_{2}x_{5} \quad \text{and} \quad z_{2}=x_{3}x_{6}-x_{4}x_{5}
\]
and consider $z_1,z_2$ as elements of $\F[\bfx_n]$. 
One can easily check that $z_1,z_2\in P(\g)$. 

We claim that 
the fraction fields of $P(\mathfrak{g})$ and 
$P_{p}(\mathfrak{g})[z_{1}, z_{2}]$ coincide.
Note first that  $z_{1}\not\in P_{p}(\mathfrak{g})[z_{2}]$, 
since for $f\in P_{p}(\mathfrak{g})[z_{2}]$ we have that  
$\partial f/\partial x_{2}=0$, but $\partial z_{1}/\partial x_{2}=-2x_{5}$. Moreover,
\[
z_{1}\in P_{p}(\mathfrak{g})[x_{2}, x_{3}]\quad\mbox{and}\quad
z_{2}\in P_{p}(\mathfrak{g})[x_{3}, x_{4}].
\]
Hence $P_{p}(\mathfrak{g})[z_{1}]\subset P_{p}(\mathfrak{g})[z_{1}, z_{2}]$ is 
a proper inclusion of rings such that $L_{p}(\mathfrak{g})(z_{1})\subset 
L_{p}(\mathfrak{g})(z_{1}, z_{2})$.  
Observe for $i=1,2$ that $z_{i}\not\in L_{p}(\mathfrak{g})$ and $z_{i}^{p}\in L_{p}(\mathfrak{g})$. As $L_{p}(\mathfrak{g})\subset \mathbb{F}(\mathbf{x}_{n})$ is 
purely inseparable, we obtain from Corollary~\ref{cor:insep} that $L_{p}(\mathfrak{g})\subset L_{p}(\mathfrak{g})(z_{i})$ is also purely inseparable and 
$\dim_{L_{p}(\mathfrak{g})}L_{p}(\mathfrak{g})(z_{i})=p$. 
Since  $\dim_{L_{p}(\mathfrak{g})}L(\mathfrak{g})=p^{2}$,
it follows that $L(\mathfrak{g})=L_{p}(\mathfrak{g})(z_{1}, z_{2})$. 

Let us finally show that $P(\g)$ 
is the integral closure of $P_{p}(\mathfrak{g})[z_{1}, z_{2}]$ in its fraction
field.
Let $S$ be the integral closure of $P_{p}(\mathfrak{g})[z_{1}, z_{2}]$ in 
$L_{p}(\mathfrak{g})(z_{1}, z_{2})$. Let $y\in S$. 
Then $y$ is integral also over  $P(\mathfrak{g})$ and $y\in L(\mathfrak{g})$. 
The algebra $P(\mathfrak{g})$ is integrally closed (Theorem~\ref{th:int-closed}), and 
so $y\in P(\mathfrak{g})$. Thus $S\subseteq  P(\mathfrak{g})$. 
For the other inclusion, recall that $P_{p}(\mathfrak{g})[z_{1}, z_{2}]\subseteq P(\mathfrak{g})$ is an integral extension. Moreover,
 $P(\mathfrak{g})\subseteq L(\mathfrak{g})=L_{p}(\mathfrak{g})(z_{1}, z_{2})$. 
 Hence $P(\mathfrak{g})\subseteq S$, and so $P(\mathfrak{g})=S$.

 By the argument in the proof of Theorem~\ref{th1}, we have  
 \[
 Z_{p}(\mathfrak{g})[z_{1}, z_{2}]\cong P_{p}(\mathfrak{g})[z_{1}, z_{2}]
 \]
 since these domains can be presented by the same generators and relations. 
 Theorem~\ref{th1} implies that $Z(\mathfrak{g})$ is the integral closure of 
 $Z_{p}(\mathfrak{g})[z_{1}, z_{2}]$ in its fraction field. Thus,  $Z(\mathfrak{g})\cong P(\mathfrak{g})$.

\section{Tables}\label{sec:tables}

In this final section we present some tables that summarise information concerning the nilpotent Lie algebras of dimension at most $6$. 
Tables~\ref{table:dim5} and~\ref{table:dim6} contain the isomorphism types for such Lie algebras according to~\cite{cgs12,cgscorr}. 
The columns contain the following information.
\begin{description}
  \item[$\g$] The name of the Lie algebras as it appears in~\cite{cgs12}.
  \item[relations] The nonzero brackets between the basis elements.
  \item[$C(\g)$] Generators for the center of $\g$.
  \item[$\mbox{cl}(\g)$] The nilpotency class of $\g$.
  \item[$r(\g)$] The value of $r(\g)=\mbox{rank}\,M(\g)$ as defined in Section~\ref{sec:pol_inv}.
\end{description}

Some Lie algebras in Table~\ref{table:dim6} depend on a parameter $\varepsilon$. The possible isomorphisms between Lie algebras 
in the same parametric family with different parameters are described in~\cite{cgs12,cgscorr}.
Over fields of characteristic two, we only consider the Lie algebras $\g^{(2)}_{6,7}(\varepsilon)$ since the nilpotency class of the
others is bigger than two.

\begin{table}[h]
$$
  \begin{array}{|l|l|l|l|l|}
  \hline
  \g & \mbox{relations} & C(\g) & \mbox{cl}(\g) & r(\g)\\
  \hline\hline
  \g_3 & [x_{1}, x_{2}]=x_{3} & x_3 & 2 & 2\\
  \hline
  \g_4 & [x_{1}, x_{2}]=x_{3}, [x_{1}, x_{3}]=x_{4} & x_4 &  3 & 2\\
  \hline
  \g_{5,1} & [x_{1}, x_{2}]=x_{5}, [x_{3}, x_{4}]=x_{5} & x_5 & 2 & 4\\
  \hline
  \g_{5,2} & [x_{1}, x_{2}]=x_{4}, [x_{1}, x_{3}]=x_{5} & x_4,x_5 & 2 & 2\\
  \hline
  \g_{5,3} & [x_{1}, x_{2}]=x_{4}, [x_{1}, x_{4}]=x_{5}, [x_{2}, x_{3}]=x_{5} 
    & x_5 & 3 & 4 \\
    \hline
  \g_{5,4} & [x_{1}, x_{2}]=x_{3}, [x_{1}, x_{3}]=x_{4}, [x_{2}, x_{3}]=x_{5} 
  & x_4,x_5 & 3 & 2 \\
  \hline  
  \g_{5,5} & [x_{1}, x_{2}]=x_{3}, [x_{1}, x_{3}]=x_{4}, [x_{1}, x_{4}]=x_{5}
 & x_5  & 4 & 2\\
    \hline
  \g_{5,6} & [x_{1}, x_{2}]=x_{3}, [x_{1}, x_{3}]=x_{4}, [x_{1}, x_{4}]=x_{5}, [x_{2}, x_{3}]=x_{5} & x_5 & 4 & 4\\
  \hline
\end{array}
$$
\caption{The nilpotent Lie algebras of dimension at most 5}
\label{table:dim5}
\end{table}

  \begin{longtable}{|l|l|l|l|l|}
    \hline
    $\g$ & \mbox{relations} & $C(\g)$ &  $\mbox{cl}(\g)$ & $r(\g)$\\
    \hline\hline
    $\g^{(2)}_{6,7}(\varepsilon)$ & \makecell[l]{$[x_1,x_2]=x_5$, $[x_1,x_3]=x_6$, $[x_2,x_4]=\varepsilon x_6$,\\ {}$[x_3,x_4]=x_5+x_6$} & $x_5$, $x_6$ & 2 & 4\\
    \hline
    $\g_{6,10}$ & $[x_{1}, x_{2}]=x_{3}$, $[x_{1} , x_{3}]=x_{6}$, $[x_{4} , x_{5}]=x_{6}$ 
    & $x_{6}$ &  3 & 4 \\
    \hline
    $\g_{6,11}$ & \makecell[l]{$[x_{1},x_{2}]=x_{3}$, $[x_{1} , x_{3}]=x_{4}$, \\
      {}$[x_{1} , x_{4}]=x_{6}$, $[x_{2} , x_{3}]=x_{6}$, $[x_{2}, x_{5}]=x_{6}$} 
      & $x_6$ &  4 & 4\\
    \hline
    $\g_{6,12}$ & \makecell[l]{$[x_{1} , x_{2}]=x_{3}$, $[x_{1} , x_{3}]=x_{4}$, 
    $[x_{1}, x_{4}]=x_{6}$, \\{}
    $[x_{2} , x_{5}]=x_{6}$} & $x_6$ & 4 & 4\\
    \hline
    $\g_{6,13}$ & \makecell[l]{$[x_{1} , x_{2}]=x_{3}$, $[x_{1} , x_{3}]=x_{5}$, 
    $[x_{1} , x_{5}]=x_{6}$,\\{} $[x_{2}, x_{4}]=x_{5}$, $[x_{3}, x_{4}]=x_{6}$} & $x_6$ &  4 & 4\\
    \hline
    $\g_{6,14}$ & \makecell[l]{$[x_{1}, x_{2}]=x_{3}$, $[x_{1}, x_{3}]=x_{4}$, 
    $[x_{1}, x_{4}]=x_{5}$, \\{}$[x_{2}, x_{3}]=x_{5}$, $[x_{2}, x_{5}]=x_{6}$, 
    $[x_{3}, x_{4}]=-x_{6}$} & $x_6$ &  5 & 4\\
    \hline
    $\g_{6,15}$ & \makecell[l]{$[x_{1}, x_{2}]=x_{3}$, $[x_{1}, x_{3}]=x_{4}$, 
    $[x_{1}, x_{4}]=x_{5}$,\\{} 
    $[x_{1}, x_{5}]=x_{6}$, $[x_{2}, x_{3}]=x_{5}$, $[x_{2}, x_{4}]=x_{6}$} 
      & $x_6$ &  5 & 4 \\
      \hline
    $\g_{6,16}$ & \makecell[l]{$[x_{1}, x_{2}]=x_{3}$, $[x_{1}, x_{3}]=x_{4}$, 
    $[x_{1}, x_{4}]=x_{5}$, \\{}
    $[x_{2}, x_{5}]=x_{6}$, $[x_{3}, x_{4}]=-x_{6}$}
   & $x_6$ &  5 & 4\\
      \hline
    $\g_{6,17}$ & \makecell[l]{$[x_{1}, x_{2}]=x_{3}$, $[x_{1}, x_{3}]=x_{4}$, 
    $[x_{1}, x_{4}]=x_{5}$, \\{}$[x_{1}, x_{5}]=x_{6}$, $[x_{2}, x_{3}]=x_{6}$} & 
    $x_6$ & 5 & 4\\
    \hline
    $\g_{6,18}$ & \makecell[l]{$[x_{1}, x_{2}]=x_{3}$, $[x_{1}, x_{3}]=x_{4}$, 
    $[x_{1}, x_{4}]=x_{5}$, \\{}$[x_{1}, x_{5}]=x_{6}$} & $x_{6}$ & 5&2 \\
    \hline
    \makecell[l]{$\g_{6,19}(\varepsilon)$\\$\varepsilon\in\F^*$}
     &\makecell[l]{$[x_{1}, x_{2}]=x_{4}$, $[x_{1}, x_{3}]=x_{5}$, $[x_{1}, x_{5}]=x_{6}$,\\{} 
     $[x_{2}, x_{4}]=x_{6}$, $[x_{3}, x_{5}]=\varepsilon x_{6}$}
    &$x_{6}$&3 &4\\
    \hline
    $\g_{6,20}$ & \makecell[l]{$[x_{1} , x_{2}]=x_{4}$, $[x_{1}, x_{3}]=x_{5}$, 
    $[x_{1}, x_{5}]=x_{6}$,\\{} $[x_{2}, x_{4}]=x_{6}$} &$x_{6}$&3&4\\
    \hline
    \makecell[l]{$\g_{6,21}(\varepsilon)$\\ $\varepsilon\in\F^*$} & 
    \makecell[l]{$[x_{1}, x_{2}]=x_{3}$, $[x_{1}, x_{3}]=x_{4}$, $[x_{1}, x_{4}]=x_{6}$,\\{} $[x_{2}, x_{3}]=x_{5}$, $[x_{2}, x_{5}]=\varepsilon x_{6}$} 
    &$x_{6}$& 4&4\\
    \hline
    \makecell[l]{$\g_{6,22}(\varepsilon)$\\ 
    $\varepsilon\in\F$} & \makecell[l]{$[x_{1}, x_{2}]=x_{5}$, $[x_{1}, x_{3}]=x_{6}$, $[x_{2}, x_{4}]=\varepsilon x_{6}$,\\{} $[x_{3}, x_{4}]=x_{5}$} &$x_{5}$, $x_{6}$&2&4\\
    \hline
    $\g_{6,23}$ & \makecell[l]{$[x_{1}, x_{2}]=x_{3}$, $[x_{1}, x_{3}]=x_{5}$, $[x_{1}, x_{4}]=x_{6}$,\\{} $[x_{2}, x_{4}]=x_{5}$} &$x_{5}$, $x_{6}$&3&4\\
    \hline
    \makecell[l]{$\g_{6,24}(\varepsilon)$\\ $\varepsilon\in\F$}& 
    \makecell[l]{$[x_{1}, x_{2}]=x_{3}$, $[x_{1}, x_{3}]=x_{5}$, $[x_{1}, x_{4}]=\varepsilon x_{6}$,\\{} $[x_{2}, x_{3}]=x_{6}$, $[x_{2}, x_{4}]=x_{5}$} &$x_{5}$, $x_{6}$&3&4\\
    \hline
    $\g_{6,25}$ & $[x_{1}, x_{2}]=x_{3}$, $[x_{1}, x_{3}]=x_{5}$, $[x_{1}, x_{4}]=x_{6}$
    &$x_{5}$, $x_{6}$& 3&2\\
    \hline
    $\g_{6,26}$ & $[x_{1}, x_{2}]=x_{4}$, $[x_{1}, x_{3}]=x_{5}$, $[x_{2}, x_{3}]=x_{6}$
    &$x_{4}$, $x_{5}$, $x_{6}$&2&2\\
    \hline
    $\g_{6,27}$ & $[x_{1}, x_{2}]=x_{3}$, $[x_{1}, x_{3}]=x_{5}$, $[x_{2}, x_{4}]=x_{6}$ &$x_{5}$, $x_{6}$&3&4\\
    \hline
    $\g_{6,28}$ & \makecell[l]{$[x_{1}, x_{2}]=x_{3}$, $[x_{1}, x_{3}]=x_{4}$, $[x_{1}, x_{4}]=x_{5}$,\\{} $[x_{2}, x_{3}]=x_{6}$} &$x_{5}$, $x_{6}$&4&4\\
    \hline
  \caption{The nilpotent Lie algebras of dimension 6}
  \label{table:dim6}
  \end{longtable}

Table~\ref{table:Zgpgen} contains the generators used in Theorem~\ref{th1} for each Lie algebra $\g$ appearing in Tables~\ref{table:dim5}
and~\ref{table:dim6}.

\begin{longtable}{|l|l|}
  \hline
  $\g$ & the generators $z_i$ \\
  \hline\hline
  $\g_4$ & $x_3^2-2x_2x_4$ \\
  \hline
  $\g_{5,2}$ & $x_2x_5-x_3x_4$ \\
  \hline
  $\g_{5,4}$ & $x_3^2 + 2x_1x_5 - 2x_2 x_4$ \\
  \hline  
  $\g_{5,5}$ & \makecell[l]{$x_4^2-2x_3 x_5$,\\
  $3x_2 x_5^2 - 3x_3 x_4 x_5 + x_4^3$}\\
    \hline
    $\g_{6,10}$ & $x_{3}^{2}-2x_{2}x_{6}$ \\
    \hline
    $\g_{6,11}$ & $x_{4}^{2}+2x_{5}x_{6}-2x_{3}x_{6}$ \\
    \hline
    $\g_{6,12}$ & $x_{4}^{2}-2x_{3}x_{6}$ \\
    \hline
    $\g_{6,13}$ & $x_{5}^{3}-3x_{3}x_{5}x_{6}+3x_{2}x_{6}^{2}$ \\
    \hline
    $\g_{6,14}$ & $2x_{5}^{3}+3x_{4}^{2}x_{6}-6x_{3}x_{5}x_{6}-6x_{1}x_{6}^{2}$ \\
    \hline
    $\g_{6,15}$ & $x_{5}^{3}-3x_{4}x_{5}x_{6}+3x_{3}x_{6}^{2}$ \\
      \hline
    $\g_{6,16}$ & $x_{4}^{2}-2x_{1}x_{6}-2x_{3}x_{5}$ \\
      \hline
    $\g_{6,17}$ & $x_{5}^{2}-2x_{4}x_{6}$ \\
    \hline
    $\g_{6,18}$ & \makecell[l]{$x_4^2+2x_2x_6-2x_3x_5$,\\ 
            $x_5^2-2x_4x_6$,\\
            $x_5^3-3x_4x_5x_6+3x_3x_6^2$}\\
    \hline
    \makecell[l]{$\g_{6,19}(\varepsilon)$\\$\varepsilon\in\F^*$}
     &$x_{5}^{2}+\varepsilon x_{4}^{2}+2\varepsilon x_{1}x_{6}-2x_{3}x_{6}$\\
    \hline
    $\g_{6,20}$ & $x_{5}^{2}-2x_{3}x_{6}$ \\
    \hline
    \makecell[l]{$\g_{6,21}(\varepsilon)$\\ $\varepsilon\in\F^*$} & 
    $x_{5}^{2}+\varepsilon x_{4}^{2}-2\varepsilon x_{3}x_{6}$ \\
    \hline
    $\g_{6,25}$ & \makecell[l]{$x_{3}^{2}-2x_{2}x_{5}$,\\ $x_{3}x_{6}-x_{4}x_{5}$}\\
    \hline
    $\g_{6,26}$ & $x_{1}x_{6}-x_{2}x_{5}+x_{3}x_{4}$\\
    \hline  
    \caption{The generators of $Z(\g)$ over $Z_p(\g)$ in large enough
    characteristic}
    \label{table:Zgpgen}
\end{longtable}



\end{document}